\documentclass[11pt]{amsart}
\usepackage{mathrsfs}
\usepackage{longtable,booktabs,needspace}
\newcommand{\qprf}{\operatorname{qprf}}
\newcommand{\Apr}{A_{\mathrm{pr}}}

\usepackage{latexsym}
\usepackage{amssymb}
\usepackage{amsfonts}

\usepackage{amscd}
\usepackage{bbm} 
\usepackage{mathabx}
\usepackage{soul}

\usepackage{color}

\usepackage{hyperref}
\hypersetup{
    pdftitle={Three-dimensional noncommutative Fano minifolds},
    pdfauthor={Konstantin Loginov and Dmitri Orlov},
    colorlinks,
    citecolor=black,
    filecolor=black,
    linkcolor=black,
    urlcolor=black
}

\usepackage[pdf,color,all]{xy}

\newtheorem{proposition}{Proposition}[section]

\newtheorem{introtheorem}{Theorem}

\newtheorem{introproposition}[introtheorem]{Proposition}
\newtheorem{lemma}[proposition]{Lemma}
\newtheorem{definition}[proposition]{Definition}
\newtheorem{theorem}[proposition]{Theorem}

\newtheorem{corollary}[proposition]{Corollary}

\newtheorem{remark}[proposition]{Remark}

\newtheorem{lemma-definition}[proposition]{Lemma-Definition}

\newcounter{tmp}

\def\A{{\mathcal A}}
\def\B{{\mathcal B}}
\def\C{{\mathcal C}}

\def\D{{\mathcal D}}

\def\R{{\mathcal R}}
\def\L{{\mathcal L}}

\def\cS{{\mathcal S}}

\def\cZ{{\mathcal Z}}

\def\ZZ{{\mathbb Z}}

\def\br{{\mathbf r}}

\def\ZZ{{\mathbb Z}}
\def\QQ{{\mathbb Q}}

\def\Hom{\operatorname{Hom}}

\def\Ext{\operatorname{Ext}}

\def\Im{\operatorname{Im}\,}

\def\kk{{\Bbbk}}

\def\QGr{\operatorname{QGr}}
\def\Tors{\operatorname{Tors}\!\operatorname{--}\!}

\def\mM{\mathsf M}

\def\mP{\mathsf P}
\def\barP{\overline{\mP}}

\def\mS{\mathsf S}

\def\m0{\mathsf 0}

\def\rr{\mathrm{r}}
\def\rs{\mathrm{s}}

\def\wt{\widetilde}

\def\Md{\operatorname{Mod}\!\operatorname{--}\!}

{\endgroup\hfill$\Box$}

{\endgroup\hfill$\Box$}

\def\mod{\operatorname{mod}\!}

\def\br{\bar}
\def\wb{\widebar}

\title[]{Three-dimensional noncommutative  Fano minifolds}

\author[]{Konstantin Loginov}
\address{Steklov Institute of Mathematics,
Gubkina str. 8, Moscow, Russia}
\email{loginov@mi-ras.ru}

\author[]{Dmitri Orlov}

\address{Steklov Institute of Mathematics,
Gubkina str. 8, Moscow, Russia}
\email{orlov@mi-ras.ru}

\date{}
\dedicatory{}

\keywords{Non-commutative algebraic geometry}

\makeatletter
\@namedef{subjclassname@2020}{\textup{2020} Mathematics Subject Classification}
\makeatother
\subjclass[2020]{14A22}

\begin{document}

\begin{abstract}
We give a construction of the previously announced family of three-dimensional noncommutative Fano minifolds of Euler degree $72.$
We prove that their defining $\ZZ$\!--algebras are Koszul and Artin--Schelter regular, and compute the Hilbert series of their section algebras.
\end{abstract}

\maketitle
\section*{Introduction}
We give a construction of the family $\mathsf M_{72}(p,q)$ of three-dimensional noncommutative Fano minifolds of Euler degree $72.$
Announced in \cite{Orl26} and recalled in \cite[Section 7.2]{LO}, this family gives the first example of a minifold whose Euler form is not realized by any commutative three-dimensional minifold.
Its existence extends the geometry of Fano minifolds beyond the commutative classification.

Full exceptional collections generating geometric helices on three-dimensional minifolds have Euler matrices in five mutation classes, distinguished by the Euler degrees $22,40,54,64,72$ \cite[Theorems I and II]{LO}.
The first four are realized by $X_{22},\; V_5,\; Q^3,$ and $\mathbb P^3,$ respectively.
The family $\mathsf M_{72}(p,q)$ realizes the fifth class.
No member of this family can be obtained from a commutative minifold by a deformation.
The commutative examples have four distinguished representatives: the three standard minifolds $\mathbb P^3,\; Q^3,\; V_5,$ and the Mukai--Umemura threefold in the family $X_{22}$ \cite{MU83}.
The member $\mathsf M_{72}(1,1)$ adds a noncommutative example to this list.

Let $\kk$ be an algebraically closed field of characteristic zero and let $p,q\in\kk^*.$
Write $\L=\L(p,q)$ for the quadratic $\ZZ$\!--algebra defined in \eqref{defalg} by the potential $\mathsf\Phi_{p,q}$ of \eqref{eq:potential}, and $\widetilde R$ for its even section algebra defined in \eqref{eq:section-algebra}.
We prove the following results, stated in \cite[Theorem 7.5, Proposition 7.6 and the subsequent discussion]{LO}, \cite[Theorems 2--5]{Orl26}.

\begin{introtheorem}\label{stmt:intro-A}
The algebra $\L$ is Koszul and AS-regular of dimension $4$ and Gorenstein parameter~$4.$
\end{introtheorem}

\begin{introproposition}\label{stmt:intro-B}
The algebra $\widetilde R$ is quadratic, Koszul, noetherian, AS-Gorenstein, and of polynomial growth, with graded component dimensions given by
\[
 \dim_{\kk}\widetilde R_m=3/2\,(m+1)^3-1/2\,(m+1),\qquad m\ge0.
\]
\end{introproposition}

Put $\D=\mathcal D(\QGr\L)^c=\qprf\!\operatorname{-}\L$ with its natural DG enhancement.
Write $\barP_i=\pi(e_i\L)\in\D,$ where $\pi:\Md\L\to\QGr\L$ is the quotient functor defined in Section \ref{subsec:z-algebras-tails}.
Let $\theta$ be the degree-two periodicity of $\L$ defined in Section \ref{subsec:defining-algebra}.

\begin{introtheorem}\label{stmt:intro-C}
The category $\D$ is equivalent to the category of perfect complexes on a smooth proper three-dimensional noncommutative Fano minifold $\mathsf M_{72}(p,q).$
The objects
$
 \mathcal P=(\barP_3,\barP_2,\barP_1,\barP_0)
$
form a full strong exceptional collection and generate a geometric helix.
The Euler matrix of $\mathcal P$ is
\[
 \Apr=\begin{pmatrix}
 1&3&11&18\\0&1&6&11\\0&0&1&3\\0&0&0&1
 \end{pmatrix}.
\]
Let $S$ be the Serre functor of $\D.$
There are isomorphisms $S(\barP_i)\simeq\barP_{i+4}[3]$ under which $S(f)$ corresponds to $\theta^2(f)[3]$ for every morphism $f:\barP_i\to\barP_j.$
The anticanonical algebra with respect to $\barP_0$ is $R=\widetilde R^{(2)}.$
The automorphism $T_*$ induced by $T=S^{-1}[3]$ is maximally unipotent on $K_0(\D)\otimes\mathbb Q,$ and the Euler degree of $\mathsf M_{72}(p,q)$ is $72.$
\end{introtheorem}

Theorem \ref{stmt:intro-A}, Proposition \ref{stmt:intro-B}, and Theorem \ref{stmt:intro-C} are proved in Theorems \ref{thm:koszul} and \ref{thm:regular}, Propositions \ref{prop:noetherian} and \ref{prop:section-algebras}, and Theorem \ref{thm:minifold}, respectively.

The authors used OpenAI's ChatGPT to investigate possible proofs, examine preliminary arguments, verify computations, and locate relevant literature.
They independently checked the mathematical arguments and prepared the final text, for which they accept full responsibility.

\section{Preliminaries}
\subsection{\texorpdfstring{$\ZZ$\!--algebras and tails}{Z-algebras and tails}}\label{subsec:z-algebras-tails}
We fix the notation for $\ZZ$\!--algebras, their module categories, and their quotient categories by torsion modules.
For background on $\ZZ$\!--algebras, quadratic duality, and helices, we refer to \cite{BP} and \cite{VB}.

A $\ZZ$\!--algebra is a $\kk$\!--linear category with objects indexed by $\ZZ.$
We write $\A_i^j=\Hom_{\A}(i,j),$ so multiplication is
$\A_j^k\otimes\A_i^j\to\A_i^k,$ and denote the local identity at $i$ by $e_i.$
A right module is a family $(\mM_i)_{i\in\ZZ}$ with actions
$\mM_j\otimes\A_i^j\to\mM_i.$
The category $\Md\A$ has projective generators $\mP_i=e_i\A,$ with
$\Hom_{\A}(\mP_i,\mM)=\mM_i$ and $\Hom_{\A}(\mP_i,\mP_j)=\A_i^j.$
A graded algebra $B$ determines the $\ZZ$\!--algebra $\A_i^j=B_{i-j},$ and $\Md\A$ is then the category of graded right $B$\!--modules.

The algebra $\A$ is \emph{connected} if its components are finite-dimensional, $\A_i^i=\kk,$ and $\A_i^j=0$ for $i<j.$
It is \emph{coconnected} if the last inequality is reversed.
Write $\mS_i$ for the simple module supported at $i.$
A periodicity of degree $r$ is an automorphism $\theta$ taking $\A_i^j$ to $\A_{i+r}^{j+r}.$
The $J$\!--Veronese subalgebra is the restriction $\A_J=\bigoplus_{i,j\in J}\A_i^j$ to an ordered subset $J\subseteq\ZZ.$

A connected algebra is \emph{generated in degree one} if it is generated by its components $\A_{i+1}^{i}$ and the local identities $e_i.$
A coconnected algebra is \emph{generated in degree $-1$} if it is generated by its components $\A_i^{i+1}$ and the local identities $e_i.$
Let $\A$ be connected and generated in degree one, or coconnected and generated in degree $-1.$
An element $m\in\mM_j$ is \emph{torsion} if $m\A_i^j=0$ for all sufficiently large $i-j$ in the connected case, or sufficiently large $j-i$ in the coconnected case.
The torsion modules form a localizing Serre subcategory $\Tors\A.$
The finite-dimensionality of the components and generation in degree $1$ or $-1$ ensure closure under extensions.
We put
\[
 \QGr\A=\Md\A/\Tors\A,\qquad
 \qprf\!\operatorname{-}\A=\mathcal D(\QGr\A)^c,
\]
where $\pi:\Md\A\to\QGr\A$ is the exact quotient functor and the superscript $c$ denotes compact objects.
We call $\QGr\A$ the \emph{category of tails} of $\A.$
The algebra $\A$ is \emph{right noetherian} if every right submodule of each $e_i\A$ is finitely generated, and \emph{left noetherian} if every left submodule of each $\A e_i$ is finitely generated, where $(e_i\A)_j=\A_j^i$ and $(\A e_i)_j=\A_i^j.$
We call $\A$ \emph{noetherian} if it is both right and left noetherian.

We use \cite[Lemma 1.16]{LO} in the following form.
\begin{proposition}\label{prop:veronese-tails}
Let $\A$ be a connected $\ZZ$\!--algebra generated in degree one, and fix integers $a$ and $r>0.$
Put $J=a+r\ZZ.$
Restriction and extension of scalars induce mutually inverse equivalences
$\QGr\A\simeq\QGr\A_J.$
The corresponding assertion holds for coconnected algebras generated in degree~$-1.$
\end{proposition}

\subsection{Quadratic algebras and duality}\label{subsec:quadratic-duality}
We fix the definitions used below.

\begin{definition}\label{def:quadratic-duality}
A connected $\ZZ$\!--algebra $\A$ is \emph{quadratic} if it is generated in degree one with defining relations in subspaces
$\R_{i+2}^{i}\subseteq\A_{i+1}^{i}\otimes\A_{i+2}^{i+1}.$
A coconnected algebra $\A$ is \emph{quadratic} if it is generated in degree $-1$ with defining relations in subspaces
$\R_i^{i+2}\subseteq\A_{i+1}^{i+2}\otimes\A_i^{i+1}.$
For a connected quadratic algebra $\A,$ its \emph{quadratic dual} $\widehat\A$ is defined by the generator and relation spaces
\[
 \widehat\A_i^{i+1}=(\A_{i+1}^{i})^*,\qquad
 \widehat\R_i^{i+2}=(\R_{i+2}^{i})^\perp
 \subseteq(\A_{i+2}^{i+1})^*\otimes(\A_{i+1}^{i})^*.
\]
Here $(\R_{i+2}^{i})^\perp$ is the annihilator of $\R_{i+2}^{i}.$
The algebra $\widehat\A$ is coconnected.
\end{definition}

\begin{definition}\label{def:koszul}
A connected quadratic $\ZZ$\!--algebra $\A$ is \emph{Koszul} if, for every $i\in\ZZ$ and $n\ge0,$ the term in homological degree $n$ of the minimal projective resolution of $\mS_i$ in $\Md\A$ is isomorphic to a direct sum of copies of $\mP_{i+n}=e_{i+n}\A.$
For a coconnected algebra, Koszulness is defined after reindexing the objects by $i\mapsto-i.$
\end{definition}
For a Koszul algebra, we have $\widehat\A_i^{i+n}=\Ext_{\A}^n(\mS_i,\mS_{i+n}),$ with multiplication given by composition of extensions.
A quadratic $\ZZ$\!--algebra is Koszul if and only if its quadratic dual is Koszul \cite[Chapter 2, Corollary 3.3]{PP}.

To define Poincar\'e--Birkhoff--Witt (PBW) bases, fix a connected quadratic $\ZZ$\!--algebra $\A$ and ordered bases
\[
 \cZ^p=\{z_1^p<\cdots<z_{d_p}^p\}
 \subset\A_{p+1}^p,\qquad d_p=\dim\A_{p+1}^p.
\]
For fixed $p$ and $n,$ order the words $z_{i_1}^p\cdots z_{i_n}^{p+n-1}$ lexicographically.

Choose a set of index pairs
$S_{p+2}^p\subseteq\{1,\ldots,d_p\}\times\{1,\ldots,d_{p+1}\}$
so that the classes of the words $z_{i_1}^pz_{i_2}^{p+1}$ with $(i_1,i_2)\in S_{p+2}^p$ form a basis of the quotient
\[
 \A_{p+2}^p=
 \bigl(\A_{p+1}^p\otimes\A_{p+2}^{p+1}\bigr)/\R_{p+2}^p
\]
and $\R_{p+2}^p$ has a basis of the following form, indexed by $\overline S_{p+2}^p,$ the set of pairs in $\{1,\ldots,d_p\}\times\{1,\ldots,d_{p+1}\}$ that do not belong to $S_{p+2}^p$:
\[
 z_{j_1}^p\otimes z_{j_2}^{p+1}
 -
 \sum_{\substack{(i_1,i_2)\in S_{p+2}^p\\
                 (i_1,i_2)<(j_1,j_2)}}
 c_{j_1j_2}^{i_1i_2}
 z_{i_1}^p\otimes z_{i_2}^{p+1}.
\]
We call the pairs in $S_{p+2}^p$ \emph{admissible}.
A word $z_{i_1}^p\cdots z_{i_n}^{p+n-1}$ is \emph{normal} if each adjacent pair is admissible:
\[
 (i_k,i_{k+1})\in S_{p+k+1}^{p+k-1},\qquad 1\le k<n.
\]

\begin{definition}
The ordered bases $\cZ^p$ are a \emph{set of PBW generators} if, for every $p\in\ZZ$ and $n\ge1,$ the images of the normal words form a basis of $\A_{p+n}^p$ under the natural projection
\[
 \A_{p+1}^p\otimes\cdots\otimes\A_{p+n}^{p+n-1}
 \longrightarrow\A_{p+n}^p.
\]
The images of the normal words, together with the local identities $e_p,$ form a \emph{PBW basis} of $\A.$
\end{definition}
\begin{definition}
A \emph{PBW $\ZZ$\!--algebra} is a quadratic $\ZZ$\!--algebra admitting a set of PBW generators.
\end{definition}

\begin{remark}\label{triplePBW}
To check the PBW property for the chosen bases, it is enough that, for every $p,$ the normal cubic words
\[
 z_{i_1}^pz_{i_2}^{p+1}z_{i_3}^{p+2},\qquad
 (i_1,i_2)\in S_{p+2}^p,\quad (i_2,i_3)\in S_{p+3}^{p+1},
\]
give a basis of $\A_{p+3}^p$ \cite[Chapter 4, Theorem 2.1 and Proposition 10.1]{PP}.
\end{remark}

For a coconnected algebra, the PBW definitions are applied after reindexing the objects by $p\mapsto-p.$
The following result is the $\ZZ$\!--algebra version of Priddy's theorem \cite[Chapter 4, Proposition 10.1]{PP}.
\begin{theorem}\label{PP}
Every PBW $\ZZ$\!--algebra is Koszul, and its quadratic dual is also a PBW $\ZZ$\!--algebra.
\end{theorem}

\begin{definition}\label{ASreg}
A connected $\ZZ$\!--algebra $\A$ is \emph{AS-regular of dimension $d$ and Gorenstein parameter $a$} if it has global dimension $d,$ the dimensions $\dim_{\kk}\A_i^j$ are bounded by a polynomial in $i-j$ for $i\ge j,$ and it satisfies
\[
 \Ext_{\A}^q(\mS_i,\mP_j)=
 \begin{cases}\kk&q=d,\ j=i+a,\\0&\text{otherwise},\end{cases}
\]
together with the analogous condition for left modules.
\end{definition}

For graded algebras, we use the following notation.
A graded algebra $B=\bigoplus_{m\ge0}B_m$ is \emph{connected} if $B_0=\kk.$
Let $\operatorname{Gr}B$ denote its category of graded right modules with morphisms of degree zero.
For connected $B,$ the projection $B\to B_0=\kk$ makes $\kk$ a graded $B$-bimodule concentrated in degree zero.
The shift $\mM(a)$ is defined by $\mM(a)_m=\mM_{m+a}.$ Thus, the bimodule $\kk(a)$ is a one-dimensional graded $B$-bimodule concentrated in degree $-a.$
The graded Ext groups are
\[
 \underline{\Ext}_B^q(\kk,B)
 =\bigoplus_{j\in\ZZ}\Ext_{\operatorname{Gr}B}^q(\kk,B(j)),
\]
where the summand indexed by $j$ is placed in degree $j.$

\begin{definition}\label{def:as-gorenstein}
The connected graded algebra $B$ is \emph{AS-Gorenstein with parameters $(d,a)$} if its left and right injective dimensions are both $d$ and
\[
 \underline{\Ext}_B^q(\kk,B)=
 \begin{cases}\kk(a)&q=d,\\0&q\ne d,\end{cases}
\]
with the same condition on the left.
\end{definition}
Note that finite global dimension is not required for AS-Gorenstein algebras.

A connected graded algebra is \emph{AS-regular of dimension $d$ and Gorenstein parameter $a$} if its global dimension is $d,$ its graded components have polynomial growth, and it is AS-Gorenstein with parameters $(d,a).$

\begin{definition}\label{def:frobenius-duality}
A coconnected $\ZZ$\!--algebra $\A$ is \emph{Frobenius of index $d$} if its components satisfy
\[
 \dim_{\kk}\A_i^{i+d}=1,\qquad
 \A_i^j=0,\qquad j-i>d,
\]
and multiplication gives perfect pairings
\[
 \A_j^{i+d}\otimes\A_i^j\longrightarrow\A_i^{i+d},
 \qquad i\le j\le i+d.
\]
\end{definition}

We use the criterion of \cite[Proposition 4.1]{BP}.
\begin{proposition}\label{prop:bp-frobenius}
A connected Koszul $\ZZ$\!--algebra $\A$ of global dimension $d$ is Gorenstein with parameter $d$ if and only if its quadratic dual $\widehat\A$ is Frobenius of index $d.$
\end{proposition}
We also use the terminology of \cite[Section 4]{BP} and \cite[Definition 8.12]{ELO}.
\begin{definition}\label{def:geometric-algebra}
A connected $\ZZ$\!--algebra is \emph{geometric of period $d$} if it is Koszul, has global dimension $d,$ and is Gorenstein with parameter $d.$
\end{definition}
Note that the geometric period need not be the least positive degree of a periodicity automorphism.

\begin{definition}\label{def:normalizing-sequence}
Let $\A=\bigoplus_{i, j\in \ZZ}\A_{i}^{j}$ be a connected $\ZZ$\!--algebra. A sequence $\mathcal T=\{t_j\in \A_{j+k}^{j},\; j\in \ZZ\}$ of elements of degree $k$  is called \emph{normalizing}
if the right ideal $\bigoplus_{ j\in \ZZ} t_j\A$ coincides with the left ideal $\bigoplus_{ j\in \ZZ} \A t_j.$
Thus, for any pair $i> j$ we have an equality $\A^{j}_{i}t_i=t_j\A^{j+k}_{i+k}$ of subspaces of $\A^{j}_{i+k}.$
Such a sequence is \emph{regular on both sides} if the maps
$\A_i^{j+k}\xrightarrow{\,t_j\,\cdot\,}\A_i^j$ and
$\A_{i-k}^{j}\xrightarrow{\,\cdot\,t_{i-k}\,}\A_i^j$
are injective for all $j\in\ZZ$ and $i\ge j+k.$
\end{definition}
For any normalizing sequence $\mathcal T$ we denote by $\langle \mathcal T\rangle$ the two-sided ideal generated by this sequence and by $\A/\langle \mathcal T\rangle$ the quotient $\ZZ$\!--algebra.
If $\A$ has an $l$\!--periodicity $\theta$ satisfying $\theta(t_j)=t_{j+l},$ then $\theta$ induces an $l$\!--periodicity of $\A/\langle\mathcal T\rangle.$

\subsection{Quantum projective planes}\label{subsec:quantum-planes}
We recall the quantum polynomial algebras in three variables and their basic properties.

For a $3\times3$ matrix $\vartheta=(\vartheta_{ij})$ with entries in $\kk^*,$ define the graded $\kk$-algebra
\begin{equation}\label{com}
 S_\vartheta=
 \kk\langle z_1,z_2,z_3\rangle/
 \langle\vartheta_{ij}z_i z_j-\vartheta_{ji}z_j z_i\mid1\le i<j\le3\rangle,
 \qquad \deg z_i=1.
\end{equation}
Write $\cS_\vartheta$ for its associated $\ZZ$\!--algebra.
Put
$
 \wt\vartheta_{ij}=\vartheta_{ij}\vartheta_{ji}^{-1}$ and $t(\vartheta)=\wt\vartheta_{12}\wt\vartheta_{23}\wt\vartheta_{31}.
$
The following properties are well known \cite[Sections 1, 6.1 and 7.2]{Rog}.
\begin{proposition}\label{prop:quantum-polynomial}
The monomials $z_1^{a_1}z_2^{a_2}z_3^{a_3},$ with $a_1,a_2,a_3\ge0,$ form a PBW basis of $S_\vartheta.$
The algebra $S_\vartheta$ is Koszul, noetherian on both sides, and AS-regular of dimension three and Gorenstein parameter three.
Its component dimensions and Hilbert series are
\[
 \dim_\kk(S_\vartheta)_m=\binom{m+2}{2},\qquad m\ge0,\qquad
 H_{S_\vartheta}(z)=(1-z)^{-3}.
\]
The quadratic dual $\widehat{\cS_\vartheta}$ is Frobenius of index three, with component dimensions
\[
 \dim_\kk\bigl(\widehat{\cS_\vartheta}\bigr)_i^{i+n}
 =\binom{3}{n},\qquad i\in\ZZ,\quad 0\le n\le3,
\]
and zero components for $n>3.$
\end{proposition}

The algebra $S_\vartheta$ depends only on the ratios $\wt\vartheta_{ij}.$
If $\vartheta'_{ij}=\vartheta_{ij}m_i$ with $m_i\ne0,$ then
the maps $(\cS_{\vartheta'})_{a+1}^{a}\to(\cS_\vartheta)_{a+1}^{a},$ given by $z_i\mapsto m_i^a z_i$ for $a\in\ZZ,$ extend to an isomorphism $\cS_{\vartheta'}\to\cS_\vartheta.$
If $t(\vartheta')=t(\vartheta),$ one can choose $m_i$ so that
$\wt\vartheta'_{ij}=\wt\vartheta_{ij}m_i/m_j.$
Indeed, choose $m_1$ arbitrarily and solve the equations for $(i,j)=(1,2),(2,3),$ after which the equation for $(3,1)$ is their consequence.
Thus, the associated $\ZZ$\!--algebras are isomorphic whenever the values of $t$ agree.

\subsection{Helices and minifolds}
For background on DG categories and smooth proper DG algebras, see \cite[Sections 2--3 and 4.7]{Ke}; for exceptional collections and helices, see \cite[Section 1]{BP}.
A smooth proper noncommutative scheme $\mathsf X$ is a DG algebra whose diagonal bimodule is perfect and whose total cohomology is finite-dimensional.
We write $\mathcal D(\mathsf X)$ for the unbounded derived category of right DG modules over $\mathsf X.$ Its category of compact objects $\D=\operatorname{Perf}(\mathsf X)$ has a Serre functor $S,$ characterized by
$\Hom(E,F)^*\cong\Hom(F,SE).$
An exceptional collection $(E_1,\ldots,E_m)$ consists of objects with
$\Hom(E_i,E_i[r])=\kk$ for $r=0$ and zero otherwise, and
$\Hom(E_j,E_i[r])=0$ for $j>i.$
The collection is \emph{full} if its objects generate $\D$ under shifts, cones, and direct summands, and \emph{strong} if $\Hom(E_i,E_j[r])=0$ for every $i,j$ and $r\ne0.$
The left and right mutations of an exceptional pair $(E,F)$ are the pairs $(L_EF,E)$ and $(F,R_FE),$
where the objects $L_EF$ and $R_FE$ are defined by the triangles
\[
 L_EF\longrightarrow\mathbf R\Hom(E,F)\otimes E\longrightarrow F,
 \qquad
 E\longrightarrow\mathbf R\Hom(E,F)^*\otimes F\longrightarrow R_FE.
\]

A \emph{helix of period $m$} is a sequence $(E_i)_{i\in\ZZ}$ whose consecutive $m$ terms are exceptional collections and for which
\[
 E_{i-m}=L_{E_{i-m+1}}\cdots L_{E_{i-1}}E_i,\qquad
 E_{i+m}=R_{E_{i+m-1}}\cdots R_{E_{i+1}}E_i.
\]
It is full if one consecutive collection is full, and \emph{geometric} if
$\Hom(E_i,E_j[r])=0$ for $i\le j$ and $r\ne0.$
For a full helix, Serre duality gives \cite[Section 1]{BP}
\begin{equation}\label{eq:helix-serre}
 E_{i+m}\simeq S^{-1}(E_i)[m-1].
\end{equation}
The connected $\ZZ$\!--algebra associated with the helix $(E_i)_{i\in\ZZ}$ has components
$\A_i^j=\Hom(E_{-i},E_{-j})$ for $i\ge j$ and zero otherwise.
Multiplication is induced by composition in $\D.$
The following result combines \cite[Theorem 4.5]{BP} with \cite[Theorem 8.14 and Proposition 8.21]{ELO}.
\begin{theorem}\label{thm:geometric-tails}
Let $\A$ be a noetherian connected geometric $\ZZ$\!--algebra of period $d,$ and put $B=\bigoplus_{0\le i,j<d}\A_i^j.$
Write $\mP_i=e_i\A$ and $\barP_i=\pi(\mP_i),$ where $\pi:\Md\A\to\QGr\A$ is the exact quotient functor.
There are equivalences
\[
 \mathcal D(\QGr\A)\simeq\mathcal D(\Md B),\qquad
 \qprf\!\operatorname{-}\A\simeq\operatorname{Perf}(B).
\]
The objects $(\barP_{d-1},\ldots,\barP_0)$ form a full strong exceptional collection, and $(\barP_{-i})_{i\in\ZZ}$ is a geometric helix whose associated $\ZZ$\!--algebra is $\A.$
\end{theorem}

Let $\mathsf X$ be a smooth proper noncommutative scheme, and fix an integer $n\ge0.$
Put $\D=\operatorname{Perf}(\mathsf X)$ and $T=S^{-1}[n],$ where $S$ is the Serre functor of $\D.$
For an exceptional object $E\in\D,$ its anticanonical algebra is
\begin{equation}\label{eq:anticanonical-definition}
 \begin{gathered}
 R(\mathsf X,E)=\bigoplus_{m\ge0}\Hom(E,T^mE),\\
 f\star g=T^r(f)\circ g,\qquad f\in R(\mathsf X,E)_m,\quad g\in R(\mathsf X,E)_r.
 \end{gathered}
\end{equation}
For $R=R(\mathsf X,E),$ let $q_R:\mathcal D(\operatorname{Gr}R)\to\mathcal D(\QGr R)$ be the functor induced by the exact quotient functor $\operatorname{Gr}R\to\QGr R.$
Following \cite[Definitions 2.1 and 2.3]{LO}, we say that the scheme $\mathsf X$ is an $n$\!--dimensional \emph{noncommutative Fano variety} if there exists an exceptional object $E\in\operatorname{Perf}(\mathsf X)$ such that
\begin{equation}\label{eq:fano-vanishing}
 \Hom(E,T^mE[r])=0,\qquad m\ge0,\quad r\ne0,
\end{equation}
and there is an equivalence $\Phi_E$ with the following restriction to compact objects:
\begin{equation}\label{eq:fano-equivalence}
 \begin{aligned}
 \Phi_E:\mathcal D(\mathsf X)\xrightarrow{\sim}&\mathcal D(\QGr R),\qquad
 \Phi_E|_{\D}:\D=\operatorname{Perf}(\mathsf X)\xrightarrow{\sim}\qprf\!\operatorname{-}R,\\
 &F\longmapsto q_R\!\left(\bigoplus_{m\ge0}\mathbf R\Hom(E,T^mF)\right).
 \end{aligned}
\end{equation}
The noncommutative Fano variety $\mathsf X$ is \emph{of geometric type} if there is an exceptional object $E\in\operatorname{Perf}(\mathsf X)$ such that \eqref{eq:fano-vanishing} holds, an equivalence \eqref{eq:fano-equivalence} exists, and the dimensions of the graded components of $R(\mathsf X,E)$ are given by a polynomial in $m$ for all $m\ge0.$
We use \cite[Proposition 2.7]{LO}.
\begin{proposition}\label{prop:helix-fano}
Let $\mathsf X$ be a smooth proper noncommutative scheme which has a full geometric helix of period $n+1.$
Put $\D=\operatorname{Perf}(\mathsf X).$
Then $\mathsf X$ is an $n$\!--dimensional noncommutative Fano variety.
Every term of the helix can be taken as the exceptional object $E$ in the definition.
If $T_*$ acts unipotently on $K_0(\D)\otimes\QQ,$ then $\mathsf X$ is of geometric type.
\end{proposition}

Following \cite[Definition 2.6]{LO}, we use the following definition.
\begin{definition}\label{def:minifold}
An $n$\!--dimensional \emph{minifold} is a smooth proper noncommutative scheme $\mathsf X$ such that the category $\D=\operatorname{Perf}(\mathsf X)$ has a full exceptional collection of length $n+1$ generating a geometric helix and $T_*$ acts maximally unipotently on $K_0(\D)\otimes\QQ.$
\end{definition}
Maximal unipotence means that $T_*-I$ has a unique nilpotent Jordan block of size $n+1.$
For a full exceptional collection $\mathcal E$ in $\operatorname{Perf}(\mathsf X),$ let
$A_{\mathcal E}=(\chi(E_i,E_j))$ be its Euler matrix.
Serre duality gives
\begin{equation}\label{eq:serre-euler-matrix}
 S_*=A_{\mathcal E}^{-1}A_{\mathcal E}^{t},\qquad
 T_*=(-1)^n A_{\mathcal E}^{-t}A_{\mathcal E}.
\end{equation}
For a three-dimensional minifold $\mathsf X,$ choose $\mathcal E$ to generate a geometric helix.
The \emph{Euler degree of $\mathsf X$} is the number $D(\mathsf X)=D(A_{\mathcal E})$ defined by
\begin{equation}\label{eq:euler-degree-definition}
 \operatorname{adj}(A_{\mathcal E}+A_{\mathcal E}^t)
   =D(A_{\mathcal E})\,vv^t,
\end{equation}
where $v$ is a primitive generator of the rank-one kernel of $A_{\mathcal E}+A_{\mathcal E}^t$ over $\ZZ,$ as in \cite[Section 3.3, formula (3.15)]{LO}.
The number $D(\mathsf X)$ is independent of the chosen collection $\mathcal E$ \cite[Proposition 3.11]{LO}.

The Euler matrices of full exceptional collections generating geometric helices on three-dimensional minifolds form five mutation classes, distinguished by their Euler degrees $22,40,54,64,72$ \cite[Theorems I and II]{LO}.
The first four classes are realized by $X_{22},\; V_5,\; Q^3,$ and $\mathbb P^3,$ respectively.
The realization of the fifth class by the noncommutative family $\mathsf M_{72}$ was announced in \cite{Orl26}.
In this article, we give a detailed construction of this family.

\section{The potential and the algebra \texorpdfstring{$\L$}{L}}
\label{sec:construction}
In this section, we construct the quadratic $\ZZ$\!--algebra $\L=\L(p,q)$ from the potential $\Phi_{p,q}$ and define its section algebras.
From now on, $\kk$ is an algebraically closed field of characteristic zero. Let~$p,q\in\kk^*.$
\subsection{Potential}
We recall the formula for the potential from \cite[Section 7.2, formula (7.4)]{LO}.
Let $V,\wb V,W,\wb W$ be three-dimensional vector spaces, and put $U=\operatorname{Sym}^2W$ and $\wb U=\operatorname{Sym}^2\wb W.$
Fix bases $x_i,\br x_i$ of $V,\wb V$ and bases $u^{ij},\br u^{ij}$ of $U,\wb U,$ with $u^{ij}=u^{ji}$ and $\br u^{ij}=\br u^{ji}.$
In $V\otimes U$ put
\begin{equation}\label{commut1}
\begin{aligned}
 \rr_1&=x_1u^{33}-p x_3u^{13},& \rr_2&=x_1u^{22}-q x_2u^{12},\\
 \rr_3&=x_2u^{33}-q x_3u^{23},& \rr_4&=x_2u^{11}-p x_1u^{12},\\
 \rr_5&=x_3u^{22}-p x_2u^{23},& \rr_6&=x_3u^{11}-q x_1u^{13},
\end{aligned}
\end{equation}
and define
\begin{equation}\label{trigon1}
 \rr_7=x_1u^{23}+x_2u^{13}+x_3u^{12}.
\end{equation}
Now we define the potential $\mathsf{\Phi}=\Phi_{p,q}\in V\otimes U\otimes\wb V\otimes\wb U$ by the following rule
\begin{equation}\label{eq:potential}
 \mathsf{\Phi}=\sum_{k=1}^3\bigl(\rr_{2k-1}\otimes\br{\rr}_{2k}
               +\rr_{2k}\otimes\br{\rr}_{2k-1}\bigr)
               +pq\,\rr_7\otimes\br{\rr}_7,
\end{equation}
where a bar on a tensor replaces each basis vector by its barred copy.
\begin{remark}\label{scalepot}
{\rm
Multiplying the basis vectors $u^{ij}$ and $\br u^{ij}$ with $i\ne j$ by $a\in\kk^*$ and leaving $u^{ii}$ and $\br u^{ii}$ fixed takes $\Phi_{p,q}$ to $\Phi_{ap,aq}.$
Thus, its isomorphism class depends only on $p/q,$ as in \cite[Remark 7.1]{LO}.
}
\end{remark}
The fixed identifications of barred and unbarred spaces define an involution on $V\otimes U\otimes\wb V\otimes\wb U$ by interchanging the first tensor factor with the third and the second with the fourth.
\begin{lemma}\label{inv}
The potential $\mathsf{\Phi}$ is invariant under this involution.
\end{lemma}
\begin{proof}
Each summand in parentheses in \eqref{eq:potential} is invariant under interchanging its two factors, as is $\rr_7\otimes\br{\rr}_7.$
\end{proof}

Contracting the potential $\mathsf\Phi$ in one tensor factor with a linear functional and cyclically permuting the remaining factors defines the maps
\begin{equation}\label{firstder}
\begin{aligned}
\nabla_{\wb U}:\wb U^*\to V\otimes U\otimes\wb V,&\qquad
\nabla_{\wb V}:\wb V^*\to\wb U\otimes V\otimes U,\\
\nabla_U:U^*\to\wb V\otimes\wb U\otimes V,&\qquad
\nabla_V:V^*\to U\otimes\wb V\otimes\wb U.
\end{aligned}
\end{equation}
\begin{lemma}\label{lem:first-contractions}
The four maps in \eqref{firstder} are injective.
\end{lemma}
\begin{proof}
Let $\epsilon_1,\epsilon_2,\epsilon_3$ be the standard basis of $\ZZ^3.$
Assign to $x_i,\br x_i$ multidegree $\epsilon_i$ and $u^{ij},\br u^{ij}$ multidegree $\epsilon_i+\epsilon_j.$
The potential $\mathsf\Phi$ is homogeneous of multidegree $(2,2,2).$
Expansion with respect to the chosen basis of $V$ gives
\[
 \mathsf\Phi=\sum_{i=1}^3x_i\otimes\nabla_V(x_i^*).
\]
The tensors $\nabla_V(x_i^*)$ in this expansion are nonzero by \eqref{eq:potential}.
The expansions with respect to the chosen bases of $U,$ $\wb V,$ and $\wb U$ also have nonzero coefficients.
The tensors $\nabla_V(x_i^*)$ and $\nabla_{\wb V}(\br x_i^*)$ have multidegree $(2,2,2)-\epsilon_i.$
The tensors $\nabla_U((u^{ij})^*)$ and $\nabla_{\wb U}((\br u^{ij})^*)$ have multidegree $(2,2,2)-\epsilon_i-\epsilon_j,$ with $i\le j.$
For each map in \eqref{firstder}, the images of the chosen basis of its domain have pairwise distinct multidegrees and are therefore linearly independent.
\end{proof}

Contraction of $\mathsf\Phi$ in two cyclically adjacent tensor factors defines
\begin{equation}\label{secondder}
\begin{aligned}
\Delta_{\wb V,\wb U}:\wb V^*\otimes\wb U^*\to V\otimes U,&\qquad
\Delta_{\wb U,V}:\wb U^*\otimes V^*\to U\otimes\wb V,\\
\Delta_{V,U}:V^*\otimes U^*\to\wb V\otimes\wb U,&\qquad
\Delta_{U,\wb V}:U^*\otimes\wb V^*\to\wb U\otimes V.
\end{aligned}
\end{equation}
In $U\otimes\wb V$ put
\begin{equation}\label{commut2}
\begin{aligned}
 \rs_1&=u^{33}\br x_1-q u^{13}\br x_3,&
 \rs_2&=u^{22}\br x_1-p u^{12}\br x_2,\\
 \rs_3&=u^{33}\br x_2-p u^{23}\br x_3,&
 \rs_4&=u^{11}\br x_2-q u^{12}\br x_1,\\
 \rs_5&=u^{22}\br x_3-q u^{23}\br x_2,&
 \rs_6&=u^{11}\br x_3-p u^{13}\br x_1,
\end{aligned}
\end{equation}
and define the seventh tensor by
\begin{equation}\label{trigon2}
 \rs_7=u^{23}\br x_1+u^{13}\br x_2+u^{12}\br x_3.
\end{equation}
Here $\br{\rs}_i\in\wb U\otimes V$ is obtained by interchanging barred and unbarred generators.
Let $c:V\otimes U\otimes\wb V\otimes\wb U\xrightarrow{\sim}U\otimes\wb V\otimes\wb U\otimes V$ be the cyclic permutation which moves the first tensor factor to the end.
Then \eqref{eq:potential} gives
\begin{equation}\label{eq:cyclic-potential}
 c(\mathsf{\Phi})=\sum_{k=1}^3\bigl(\rs_{2k-1}\otimes\br{\rs}_{2k}
                +\rs_{2k}\otimes\br{\rs}_{2k-1}\bigr)
                +pq\,\rs_7\otimes\br{\rs}_7.
\end{equation}
\begin{lemma}\label{7-dim}
The images of the four maps in \eqref{secondder} have dimension seven.
They are spanned, respectively, by the tensors $\rr_1,\ldots,\rr_7,$ $\rs_1,\ldots,\rs_7,$ $\br{\rr}_1,\ldots,\br{\rr}_7,$ and $\br{\rs}_1,\ldots,\br{\rs}_7.$
\end{lemma}
\begin{proof}
The tensors $\rr_1,\ldots,\rr_7$ have pairwise disjoint monomial supports and are therefore linearly independent.
The same argument applies to $\rs_1,\ldots,\rs_7,$ to $\br{\rr}_1,\ldots,\br{\rr}_7,$ and to $\br{\rs}_1,\ldots,\br{\rs}_7.$
The $7\times7$ matrix whose entry in row $i$ and column $j$ is the coefficient of $\rr_i\otimes\br{\rr}_j$ in \eqref{eq:potential} is $\operatorname{diag}(J,J,J,pq),$ where $J=\left(\begin{smallmatrix}0&1\\1&0\end{smallmatrix}\right).$
Equation \eqref{eq:cyclic-potential} gives the same matrix for the coefficients of $\rs_i\otimes\br{\rs}_j$ in $c(\mathsf\Phi).$
Its determinant is $-pq\ne0.$
Hence, the maps $\Delta_{\wb V,\wb U}$ and $\Delta_{\wb U,V}$ have rank seven.
By Lemma \ref{inv}, the maps $\Delta_{V,U}$ and $\Delta_{U,\wb V}$ also have rank seven.
\end{proof}

\subsection{The defining \texorpdfstring{$\ZZ$\!--algebra}{Z-algebra}}\label{subsec:defining-algebra}
We use the images of the four maps in \eqref{secondder} as the quadratic relation spaces in the definition of $\L.$

\begin{definition}
We define a connected quadratic $\ZZ$\!--algebra $\L=\bigoplus_{i, j\in \ZZ}\L_{i}^{j}$ as follows
\begin{equation}\label{defalg}
\L_{j+1}^{j}=
\begin{cases}
V & \text{if \;} j\equiv 0\, (\mod\; 4),\\
U & \text{if \;}  j\equiv 1\, (\mod\; 4),\\
\wb{V} & \text{if \;}  j\equiv 2\, (\mod\; 4),\\
\wb{U} & \text{if \;}  j\equiv 3\, (\mod\; 4),
\end{cases}
\qquad\text{with relations}
\quad
\R_{j+1}^{j-1}=
\begin{cases}
\Im\Delta_{U,\wb{V}} & \text{if \;}  j\equiv 0\, (\mod\; 4),\\
\Im\Delta_{\wb{V},\wb{U}}& \text{if \;}  j\equiv 1\, (\mod\; 4),\\
\Im\Delta_{\wb{U},V} & \text{if \;}  j\equiv 2\, (\mod\; 4),\\
\Im\Delta_{V,U}& \text{if \;}  j\equiv 3\, (\mod\; 4).
\end{cases}
\end{equation}
\end{definition}

By Lemma \ref{inv}, the identifications of barred and unbarred generators define a periodicity of degree $2$ with component maps $\theta:\L_i^j\xrightarrow{\sim}\L_{i+2}^{j+2}$ for all $i,j\in\ZZ.$
Define the graded algebra
\begin{equation}\label{eq:section-algebra}
 {\widetilde R}=\bigoplus_{m\ge0}{\widetilde R}_m,\qquad {\widetilde R}_m=\L_0^{-2m},\qquad
 a\star b=\theta^{-n}(a)b,\qquad a\in {\widetilde R}_m,\ b\in {\widetilde R}_n.
\end{equation}
The identity $\theta^{-n}(\L_0^{-2m})=\L_{-2n}^{-2m-2n}$ makes this product well-defined, and its associativity follows from that of $\L.$
We write $R={\widetilde R}^{(2)},$ with $R_m={\widetilde R}_{2m}.$

In computations we use $\theta$ to identify the generator spaces at indices of the same parity and suppress bars when no confusion is possible.
A word in the chosen generators of $\L$ always denotes a composable path, with its position determined by its endpoints.

\section{The quadratic dual and Frobenius pairings}\label{sec:dual}
In this section, we compute the quadratic dual $\widehat\L$ of $\L,$ deduce Koszulness of $\L,$ and prove that $\widehat\L$ is Frobenius of index four.

\subsection{Generators and relations}

Let $\L=\L(p,q)$ be the quadratic $\ZZ$\!--algebra defined in \eqref{defalg}, with quadratic relation spaces $\R_{i+2}^{i}.$
We denote its quadratic dual by $\widehat\L.$
Following the definition in Section \ref{subsec:quadratic-duality}, the generator spaces and quadratic relation spaces of $\widehat\L$ are
\[
 \widehat\L_i^{i+1}=(\L_{i+1}^{i})^*,\qquad
 \widehat\R_i^{i+2}=(\R_{i+2}^{i})^\perp
 \subseteq(\L_{i+2}^{i+1})^*\otimes(\L_{i+1}^{i})^*.
\]
Thus, the generator space $\widehat\L_i^{i+1}$ is $V^*,U^*,\wb V^*,\wb U^*$ for $i\equiv0,1,2,3\pmod4,$ respectively.
The two-periodicity of $\L$ induces a two-periodicity of $\widehat\L.$
By Lemma \ref{7-dim}, its quadratic relation spaces have dimension $18-7=11,$ and the components $\widehat\L_i^{i+2}$ have dimension $7$ for all $i\in\ZZ.$

Let us describe the subspaces of relations $\widehat{\R}_{j-1}^{j+1}$ in the chosen bases.
Denote by $\xi_i, \br{\xi}_i, \upsilon^{ij}, \br{\upsilon}^{ij}$ the bases of the spaces $V^*, \wb{V}^*, U^*, \wb{U}^*$ that are dual to the bases $x_i, \br{x}_i, u^{ij}, \br{u}^{ij}$ of $V, \wb{V}, U, \wb{U}.$ Thus, the 11-dimensional subspace $\widehat{\R}_{0}^{2}\subset \widehat{\L}_{1}^{2}\otimes \widehat{\L}_{0}^{1}\cong U^*\otimes V^*$ has the following basis:
\begin{equation}\label{anticommut1}
\begin{alignedat}{3}
p \upsilon^{33}\xi_1 + \upsilon^{31}\xi_3,& \quad & q\upsilon^{22}\xi_1 + \upsilon^{12}\xi_2, & \quad &\upsilon^{11}\xi_1,
\\
 q\upsilon^{33}\xi_2 + \upsilon^{23}\xi_3,& \quad &
p\upsilon^{11}\xi_2 + \upsilon^{12}\xi_1, & \quad & \upsilon^{22}\xi_2,
\\
 p\upsilon^{22}\xi_3 + \upsilon^{23}\xi_2,& \quad &
q \upsilon^{11}\xi_3 + \upsilon^{31}\xi_1, & \quad &\upsilon^{33}\xi_3,
\\
 \upsilon^{23}\xi_1- \upsilon^{31}\xi_2, & \quad & \upsilon^{31}\xi_2- \upsilon^{12}\xi_3 , & \quad &
\end{alignedat}
\end{equation}
while the subspace $\widehat{\R}_{1}^{3}\subset \widehat{\L}_{2}^{3}\otimes \widehat{\L}_{1}^{2}\cong \wb V^*\otimes U^*$ has the following basis:

\begin{equation}\label{anticommut2}
\begin{alignedat}{3}
q \br\xi_1\upsilon^{33} + \br\xi_3\upsilon^{31},& \quad & p\br\xi_1\upsilon^{22} + \br\xi_2\upsilon^{12}, & \quad &\br\xi_1\upsilon^{11},
\\
 p\br\xi_2\upsilon^{33} + \br\xi_3\upsilon^{23},& \quad &
q\br\xi_2\upsilon^{11} + \br\xi_1\upsilon^{12}, & \quad & \br\xi_2\upsilon^{22},
\\
 q\br\xi_3\upsilon^{22} + \br\xi_2\upsilon^{23},& \quad &
p\br\xi_3\upsilon^{11} + \br\xi_1\upsilon^{31}, & \quad &\br\xi_3\upsilon^{33},
\\
\br\xi_1\upsilon^{23}- \br\xi_2\upsilon^{31}, & \quad & \br\xi_2\upsilon^{31}- \br\xi_3 \upsilon^{12}. & \quad &
\end{alignedat}
\end{equation}
Since the $\ZZ$\!--algebra $\widehat\L,$ like the $\ZZ$\!--algebra $\L,$ is actually 2-periodic, the given bases of the subspaces
$\widehat{\R}_{0}^{2}$ and $\widehat{\R}_{1}^{3}$ provide bases of all subspaces $\widehat{\R}_{j-1}^{j+1}.$

\subsection{Components of the quadratic dual}\label{subsec:dual-components}
We construct a finite-dimensional graded algebra $F$ and identify $\widehat\L^{\mathrm{op}}$ with a reindexed Veronese subalgebra of the $\ZZ$\!--algebra associated with $F.$
We then deduce Koszulness of $\L$ and compute the dimensions of the components of $\widehat\L.$

Put $\rho(2m)=3m$ and $\rho(2m+1)=3m+1,$ and let $J=\rho(\ZZ).$
Write $\cS_{\vartheta,J}$ for the $J$\!--Veronese subalgebra of $\cS_\vartheta,$ reindexed by $\rho,$ so that $(\cS_{\vartheta,J})_j^i=(S_\vartheta)_{\rho(j)-\rho(i)}.$
\begin{lemma}\label{lem:quantum-veronese-pbw}
The $\ZZ$\!--algebra $\cS_{\vartheta,J}$ is PBW and hence Koszul.
\end{lemma}
\begin{proof}
By Proposition \ref{prop:quantum-polynomial}, the ordered monomials form a PBW basis of $S_\vartheta.$
Equip $(\cS_{\vartheta,J})_{i+1}^{i}=(S_\vartheta)_{\rho(i+1)-\rho(i)}$ with the lexicographically ordered monomial basis.
The proof of \cite[Chapter 4, Proposition 4.3]{PP} uses the following observation: for normal words $u_1,\ldots,u_n$ in $z_1,z_2,z_3,$ the word $u_1\cdots u_n$ is normal if and only if the words $u_r u_{r+1}$ are normal for $1\le r<n.$
This observation holds for any positive lengths of the $u_r,$ so the same proof shows that $\cS_{\vartheta,J}$ is quadratic and that the chosen bases are PBW generators.
Koszulness follows from Theorem \ref{PP}.
\end{proof}

Choose $\nu\in\kk^*$ with $\nu^3=p/q.$
Let $S_\eta$ be the homogeneous coordinate algebra of a quantum projective plane as in Section \ref{subsec:quantum-planes}, with $\wt\eta_{12}=\wt\eta_{23}=\wt\eta_{31}=\nu^{-1}.$
Denote the generators of $S_\eta$ by $y_1,y_2,y_3.$
Define the graded algebra $F=S_\eta/(y_1^3,y_2^3,y_3^3),$ explicitly
\begin{equation}\label{eq:truncated-quantum}
 F=\kk\langle y_1,y_2,y_3\rangle\big/\bigl(y_2y_1-\nu^{-1}y_1y_2,\ y_3y_2-\nu^{-1}y_2y_3,\ y_3y_1-\nu y_1y_3,\ y_1^3,y_2^3,y_3^3\bigr).
\end{equation}
The monomials $y_1^ay_2^by_3^c,$ $0\le a,b,c\le2,$ form the basis of $F$.

The $\ZZ$\!--algebra $\mathcal F$ associated to $F$ has components $\mathcal F_j^i=F_{j-i}.$
Let $\mathcal F_J$ be its $J$\!--Veronese subalgebra, reindexed by $\rho,$ so that $(\mathcal F_J)_j^i=F_{\rho(j)-\rho(i)}.$

\begin{proposition}\label{prop:dual-veronese}
Define linear maps from the generator spaces of $\widehat\L^{\mathrm{op}}$ to those of $\mathcal F_J$ by the following formulas, for $1\le a\le3$:
\begin{equation}\label{eq:dual-map}
\begin{aligned}
 \xi_a&\longmapsto y_a,&\upsilon^{aa}&\longmapsto y_a^2,\\
 \upsilon^{12}&\longmapsto-q\nu y_1y_2,&
 \upsilon^{13}&\longmapsto-q\nu^2y_1y_3,&
 \upsilon^{23}&\longmapsto-q\nu y_2y_3.
\end{aligned}
\end{equation}
The generators $\br{\xi}_a$ and $\br{\upsilon}^{ab}$ are sent to the same elements of $F$ as $\xi_a$ and $\upsilon^{ab},$ respectively.
The maps \eqref{eq:dual-map} extend to an isomorphism $\widehat\L^{\mathrm{op}}\cong\mathcal F_J.$
The component maps of the isomorphism $\widehat\L^{\mathrm{op}}\cong\mathcal F_J$ are linear isomorphisms $\psi_i^j:\widehat\L_i^j\xrightarrow{\sim}F_{\rho(j)-\rho(i)}$ for $i\le j.$
They satisfy $\psi_i^k(ba)=\psi_i^j(a)\psi_j^k(b)$ for $a\in\widehat\L_i^j$ and $b\in\widehat\L_j^k,$ where $i\le j\le k.$
\end{proposition}
\begin{proof}
By Lemma \ref{lem:quantum-veronese-pbw}, the algebra $\cS_{\eta,J}$ is quadratic.
Put $I=(y_1^3,y_2^3,y_3^3)\subseteq S_\eta.$
Since each $y_a^3$ is normal in $S_\eta,$ we have $I_d=\sum_{a=1}^3y_a^3(S_\eta)_{d-3}$ for $d\ge3,$ and $I_d=0$ for $d<3.$
Since $\rho(i+2)-\rho(i)=3,$ these identities show that the kernel of $\cS_{\eta,J}\to\mathcal F_J$ is generated by $y_a^3\in(\cS_{\eta,J})_{i+2}^i$ for $1\le a\le3$ and $i\in\ZZ.$
Thus, the algebra $\mathcal F_J$ is quadratic.

The assignments \eqref{eq:dual-map} annihilate the relation tensors in \eqref{anticommut1}--\eqref{anticommut2}, as verified in Appendix \ref{sec:calculations}, and therefore extend to a homomorphism $\widehat\L^{\mathrm{op}}\to\mathcal F_J.$
Both multiplication maps $F_1\otimes F_2\to F_3$ and $F_2\otimes F_1\to F_3$ are surjective.
Since $\dim F_3=7,$ the eleven independent relation tensors in each of \eqref{anticommut1} and \eqref{anticommut2} form the entire quadratic kernel.
Since $\widehat\L^{\mathrm{op}}$ and $\mathcal F_J$ are quadratic, the homomorphism $\widehat\L^{\mathrm{op}}\to\mathcal F_J$ defined by \eqref{eq:dual-map} is an isomorphism.
\end{proof}

\begin{theorem}\label{thm:koszul}
The quadratic $\ZZ$\!--algebra $\L$ is Koszul.
\end{theorem}
\begin{proof}
By Lemma \ref{lem:quantum-veronese-pbw}, the algebra $\cS_{\eta,J}$ is Koszul.
Put $G_a=S_\eta/(y_1^3,\ldots,y_{a-1}^3)$ for $1\le a\le4,$ so that $G_1=S_\eta$ and $G_4=F.$
Each $G_a$ has basis $y_1^{b_1}y_2^{b_2}y_3^{b_3},$ where $b_1,b_2,b_3\in\ZZ_{\ge0}$ and $b_r\le2$ for $1\le r<a.$
For $1\le a\le3,$ the image of $y_a^3$ is normal in $G_a$ by the first three relations in \eqref{eq:truncated-quantum}, and multiplication by it on either side sends distinct basis monomials to nonzero scalar multiples of distinct basis monomials, so the class of $y_a^3$ in $G_a$ is regular.
Since $\rho(i+2)-\rho(i)=3,$ the elements represented by $y_a^3\in(G_a)_{\rho(i+2)-\rho(i)},$ with $i\in\ZZ,$ form a regular normalizing sequence of degree two in the reindexed $J$\!--Veronese algebra associated with $G_a,$ with quotient the corresponding algebra associated with $G_{a+1}.$
Applying the proof of \cite[Chapter 2, Corollary 6.3]{PP} successively for $a=1,2,3,$ with the simple modules $\mS_i$ in place of $\kk,$ gives Koszulness of $\mathcal F_J.$
By Proposition \ref{prop:dual-veronese}, the algebra $\widehat\L^{\mathrm{op}}$ is therefore Koszul, and passing to the opposite algebra and applying quadratic duality gives Koszulness of $\L$ \cite[Chapter 2, Corollary 3.3]{PP}.
\end{proof}

We define the matrix of Hilbert series $\widehat H(z)$ of the quadratic dual $\widehat\L$ by
\[
 \widehat H(z)_{ab}
 =\sum_{\substack{n\ge0\\n\equiv b-a\ (2)}}
 \dim\widehat\L_a^{a+n}\,z^n,
 \qquad a,b\in\{0,1\}.
\]
The row index $a$ and column index $b$ specify the parities of $i$ and $j$ in $\widehat\L_i^j,$ respectively.

\begin{corollary}\label{prop:dual}
For every $i\in\ZZ,$ the component $\widehat\L_i^{i+4}$ is one-dimensional, and $\widehat\L_i^j=0$ whenever $j-i>4.$
For $j-i=3,$ the duals of the contraction maps in \eqref{firstder} induce natural isomorphisms
\begin{equation}\label{thirdpart}
\widehat{\L}_{i}^{j}\cong
\begin{cases}
V & \text{if \;}  j\equiv 0\, (\mod\; 4),\\
U & \text{if \;}  j\equiv 1\, (\mod\; 4),\\
\wb{V} & \text{if \;}  j\equiv 2\, (\mod\; 4),\\
\wb{U} & \text{if \;}  j\equiv 3\, (\mod\; 4).
\end{cases}
\end{equation}
The matrix $\widehat H(z)$ is given by
\begin{equation}\label{eq:dual-hilbert}
 \widehat H(z)=\begin{pmatrix}
 1+7z^2+z^4&3z+6z^3\\
 6z+3z^3&1+7z^2+z^4
 \end{pmatrix}.
\end{equation}
\end{corollary}
\begin{proof}
By Proposition \ref{prop:dual-veronese}, we have $\dim_\kk\widehat\L_i^j=\dim_\kk F_{\rho(j)-\rho(i)}.$
The monomial basis of $F$ gives $(\dim F_0,\ldots,\dim F_6)=(1,3,6,7,6,3,1)$ and $F_d=0$ for $d>6.$
Together with $\rho(i+4)-\rho(i)=6,$ this proves \eqref{eq:dual-hilbert} and the assertions about $\widehat\L_i^{i+4}$ and $\widehat\L_i^j$ for $j-i>4.$

For the isomorphisms in \eqref{thirdpart}, consider the dual map $\nabla_V^*:\wb U^*\otimes\wb V^*\otimes U^*\to V$ to the map $\nabla_V$ in \eqref{firstder}.
The image of $\nabla_V$ satisfies
\[
 \Im\nabla_V\subseteq
 \bigl(\Im\Delta_{\wb U,V}\otimes\wb U\bigr)
 \cap\bigl(U\otimes\Im\Delta_{V,U}\bigr),
\]
since both inclusions are obtained by contracting the same potential in two successive factors.
By \eqref{defalg} and Definition \ref{def:quadratic-duality}, the subspace $\widehat\R_2^4\otimes U^*$ annihilates $U\otimes\Im\Delta_{V,U},$ while $\wb U^*\otimes\widehat\R_1^3$ annihilates $\Im\Delta_{\wb U,V}\otimes\wb U.$
Thus, the map $\nabla_V^*$ vanishes on the subspace $\widehat\R_2^4\otimes U^*+\wb U^*\otimes\widehat\R_1^3.$
This subspace is exactly the kernel of the quotient map $\wb U^*\otimes\wb V^*\otimes U^*\to\widehat\L_1^4,$ so we obtain the factorization
\begin{equation}\label{eq:dual-contraction-factorization}
 \wb U^*\otimes\wb V^*\otimes U^*
\longrightarrow\widehat\L_1^4\longrightarrow V.
\end{equation}
The first arrow is the quotient map, and the composite is $\nabla_V^*.$
The second arrow is surjective because $\nabla_V$ is injective by Lemma \ref{lem:first-contractions}.
Applying this argument to $\nabla_U^*,$ $\nabla_{\wb V}^*,$ and $\nabla_{\wb U}^*$ gives the surjections
\[
 \widehat\L_2^5\longrightarrow U,\qquad
 \widehat\L_3^6\longrightarrow\wb V,\qquad
 \widehat\L_0^3\longrightarrow\wb U.
\]
Since $\dim\widehat\L_1^4=\dim\widehat\L_3^6=3$ and $\dim\widehat\L_2^5=\dim\widehat\L_0^3=6,$ all four surjections are isomorphisms.
\end{proof}

\subsection{Frobenius pairings}
We construct multiplication pairings between components of $\widehat\L$ whose degrees sum to four.

\begin{corollary}\label{cor:dual-pairing}
The quadratic dual $\widehat\L$ is Frobenius of index four.
There are functionals $\varepsilon_i:\widehat\L_i^{i+4}\to\kk$ giving perfect multiplication pairings and satisfying
\begin{equation}\label{eq:graded-trace}
 \varepsilon_i(ba)=(-1)^{(j-i)(i+4-j)}
 \varepsilon_j\bigl(\widehat\theta^{\,2}(a)b\bigr),
 \qquad a\in\widehat\L_i^j,\quad b\in\widehat\L_j^{i+4}.
\end{equation}
Here $i\le j\le i+4$ and $\widehat\theta$ is the degree-two periodicity of the dual algebra.
\end{corollary}
\begin{proof}
Let $\lambda:F_6\to\kk$ extract the coefficient of $y_1^2y_2^2y_3^2.$
For exponent vectors $\alpha,\beta,$ multiplication in $F$ is given by
\[
 y^\alpha y^\beta
 =\nu^{-\alpha_2\beta_1+\alpha_3\beta_1-\alpha_3\beta_2}y^{\alpha+\beta}
\]
when $\alpha+\beta\le(2,2,2)$ componentwise, and the product is zero otherwise.
If $\alpha+\beta=(2,2,2),$ the exponent is unchanged upon interchanging $\alpha$ and $\beta.$
A basis monomial $y^\alpha\in F_r$ has nonzero product with a basis monomial $y^\beta\in F_{6-r}$ if and only if $\beta=(2,2,2)-\alpha.$
Thus, the functional $\lambda$ gives symmetric perfect pairings between $F_r$ and $F_{6-r}.$
Define $\varepsilon_i=(-1)^i\lambda\circ\psi_i^{i+4}.$
For $r=\rho(j)-\rho(i),$ the complementary degree is $6-r=\rho(i+4)-\rho(j),$ so all pairings asserted for $\widehat\L$ are perfect.
Since $\rho(i+4)=\rho(i)+6$ and $(j-i)(i+4-j)\equiv j-i\pmod2,$ symmetry of $\lambda$ gives \eqref{eq:graded-trace}.
\end{proof}

\section{Normal quotients and quantum projective planes}\label{sec:quotients}
In this section, we identify two quotients of $\L$ with reindexed Veronese subalgebras of the $\ZZ$\!--algebras associated with quantum polynomial algebras in three variables.
We compute the Hilbert series of $\L$ and these quotients, construct a basis of $\L,$ and prove that $\L,$ $\widetilde R,$ and $R$ are noetherian on both sides.

\subsection{Normalizing sequences}
We construct two normalizing sequences of degree two in $\L$ and use them to define the quotients $\C$ and $\underline\C.$

Fix a primitive cube root of unity $\zeta\in\kk$ and a scalar $\nu\in\kk^*$ with $\nu^3=p/q.$
\begin{proposition}\label{prop:normal}
A $2$\!--periodic sequence  $\mathcal T=\{t_j\in \L_{j+2}^{j}, \; j\in \ZZ\}$ of elements of degree $2$ generated by $t_0=x_1u^{23}-\zeta x_2u^{31}$ and $t_1=u^{23}\br{x}_1-\zeta u^{12}\br{x}_3$ is normalizing.
\end{proposition}
\begin{proof}
The relations \eqref{trigon1} and \eqref{trigon2} give
\begin{equation}\label{symmetr}
\begin{aligned}
 t_0&=x_1u^{23}-\zeta x_2u^{31}
 =\zeta^{-1}(x_2u^{31}-\zeta x_3u^{12})
 =\zeta(x_3u^{12}-\zeta x_1u^{23}),\\
 t_1&=u^{23}\br x_1-\zeta u^{12}\br x_3
 =\zeta^{-1}(u^{12}\br x_3-\zeta u^{31}\br x_2)
 =\zeta(u^{31}\br x_2-\zeta u^{23}\br x_1).
\end{aligned}
\end{equation}
The relations \eqref{commut1} and \eqref{commut2} imply
\begin{equation}\label{eq:normal-preliminary}
 x_2u^{13}\br x_1=x_1u^{12}\br x_3,
 \qquad
 u^{11}t_2=\zeta^{-1}(u^{12}\br x_3-\zeta u^{13}\br x_2)\br u^{11}.
\end{equation}
For the generator $u^{23},$ we successively use the relations $\rs_5=0,$ $\br{\rr}_1=0,$ $\rs_2=0,$ and $\br{\rr}_3=0$ from \eqref{commut2} and the barred version of \eqref{commut1} to obtain
\begin{equation}\label{eq:normal-u23-calculation}
 u^{23}\br x_2\br u^{13}
 =q^{-1}u^{22}\br x_3\br u^{13}
 =(pq)^{-1}u^{22}\br x_1\br u^{33}
 =q^{-1}u^{12}\br x_2\br u^{33}
 =u^{12}\br x_3\br u^{23}.
\end{equation}
Combining \eqref{eq:normal-preliminary} and \eqref{eq:normal-u23-calculation} with \eqref{symmetr} and $t_2=\theta(t_0),$ we obtain
\begin{equation}\label{eq:normal-base-identities}
 t_0\br x_1=x_1t_1,\qquad
 u^{11}t_2=t_1\br u^{11},\qquad
 u^{23}t_2=t_1\br u^{23}.
\end{equation}
Cyclic permutation of the indices preserves the relations and takes $t_0$ to $\zeta t_0$ and $t_1$ to $\zeta^{-1}t_1$ by \eqref{symmetr}.
Applying this cyclic permutation to \eqref{eq:normal-base-identities} gives the nine identities
\begin{equation}\label{eq:normal-automorphism}
\begin{gathered}
 a\,t_{j+1}=t_j\,\theta(\gamma(a)),\qquad a\in\L_{j+1}^{j}\\[6pt]
\begin{array}{c|ccccccccc}
 a&x_1&x_2&x_3&u^{11}&u^{12}&u^{13}&u^{22}&u^{23}&u^{33}\\ \hline
 \gamma(a)/a&1&\zeta^2&\zeta&1&\zeta^2&\zeta&\zeta&1&\zeta^2
\end{array}.
\end{gathered}
\end{equation}
The diagonal map $\gamma$ is defined by the table in \eqref{eq:normal-automorphism}.
It preserves all defining relations and fixes each $t_j.$
Thus, it is an automorphism of $\L,$ and \eqref{eq:normal-automorphism} proves that $\mathcal T$ is normalizing.
\end{proof}

We can also take the other primitive cube root of unity $\zeta^{-1}=\wb{\zeta}$ and consider another $2$\!--periodic normalizing sequence $\underline{\mathcal T}$ that is generated by
the elements $\underline{t}_0=x_1u^{23}-\zeta^{-1} x_2u^{31}$ and $\underline{t}_1=u^{23}\br{x}_1-\zeta^{-1} u^{12}\br{x}_3.$

\begin{definition} We denote by $\C=\bigoplus_{i, j\in \ZZ}\C_{i}^{j}$ and $\underline{\C}=\bigoplus_{i, j\in \ZZ}{\underline{\C}}_{i}^{j}$ the quotient $\ZZ$\!--algebras $\L/\langle \mathcal T\rangle$ and $\L/\langle \underline{\mathcal T}\rangle,$ respectively.
\end{definition}

\subsection{Identification of the quotient algebras}
We identify $\C$ and $\underline\C$ with Veronese subalgebras of the $\ZZ$\!--algebras associated with the quantum polynomial algebras of Section \ref{subsec:quantum-planes}.

We use the reindexed $J$\!--Veronese algebras $\cS_{\vartheta,J}$ defined in Section \ref{subsec:dual-components}.
The equations $t_0=t_1=0$ in $\C$ give
\begin{equation}\label{quotienteq}
 x_2u^{31}=\zeta^{-1}x_1u^{23},\quad x_3u^{12}=\zeta x_1u^{23},\qquad
 u^{23}\br x_1=\zeta u^{12}\br x_3,\quad u^{31}\br x_2=\zeta^{-1}u^{12}\br x_3.
\end{equation}

\begin{theorem}\label{thm:quantum-plane}
The $\ZZ$\!--algebras $\C$ and $\underline\C$ are isomorphic to the $\ZZ$\!-algebras $\cS_{\vartheta, J}$ and $\cS_{\underline\vartheta, J},$
where $\vartheta,\; \underline\vartheta$ are $3\times 3$ matrices such that $t(\vartheta)=\zeta^{-1}q^{-1}p$ and $t(\underline\vartheta)=\zeta q^{-1}p,$ respectively.
\end{theorem}
\begin{proof}
Choose the representative
\[
 \wt\vartheta_{12}=\zeta\nu,\qquad
 \wt\vartheta_{23}=\nu,\qquad
 \wt\vartheta_{31}=\zeta\nu.
\]
Then $t(\vartheta)=\zeta^{-1}\nu^3=\zeta^{-1}p/q,$ and the defining relations \eqref{com} of $S_\vartheta$ take the form
\begin{equation}\label{eq:quantum-plane-rules}
 z_2z_1=\zeta\nu z_1z_2,\qquad
 z_3z_2=\nu z_2z_3,\qquad
 z_3z_1=\zeta^2\nu^{-1}z_1z_3.
\end{equation}
For each $j\in\ZZ,$ define a linear map $\C_{j+1}^{j}\to(\cS_{\vartheta,J})_{j+1}^{j}=(S_\vartheta)_{\rho(j+1)-\rho(j)}$ by
\begin{equation}\label{eq:plane-map}
\begin{aligned}
 x_i,\br x_i&\longmapsto z_i,&u^{ii},\br u^{ii}&\longmapsto z_i^2,\\
 u^{12},\br u^{12}&\longmapsto\zeta^2\nu^2p^{-1}z_1z_2,&
 u^{13},\br u^{13}&\longmapsto\zeta\nu p^{-1}z_1z_3,&
 u^{23},\br u^{23}&\longmapsto\nu^2p^{-1}z_2z_3.
\end{aligned}
\end{equation}
Substitution in \eqref{commut1}, \eqref{commut2}, and \eqref{quotienteq} uses $\nu^3=p/q$ and gives zero for every relation.
For instance, $x_1u^{33}=px_3u^{13}$ becomes
$z_1z_3^2=\zeta\nu z_3z_1z_3,$ and
$u^{22}x_1=pu^{12}x_2$ becomes
$z_2^2z_1=\zeta^2\nu^2z_1z_2^2.$
Both equalities follow from \eqref{eq:quantum-plane-rules}.
The relation $x_2u^{13}=\zeta^2x_1u^{23}$ follows by moving $z_2$ past $z_1.$
For the remaining relations in \eqref{commut1}, \eqref{commut2}, and \eqref{quotienteq}, we substitute \eqref{eq:plane-map} and rewrite both cubic products in the order $z_1,z_2,z_3$ using \eqref{eq:quantum-plane-rules}.
The resulting coefficients agree because $\nu^3=p/q.$
The maps \eqref{eq:plane-map} are isomorphisms on the generator spaces.
The tensor representatives of $t_0,t_1$ from Proposition \ref{prop:normal} are independent of $\rr_1,\ldots,\rr_7$ and $\rs_1,\ldots,\rs_7,$ respectively: each has two of the three monomials in \eqref{trigon1} and \eqref{trigon2}, respectively, while the tensors in \eqref{commut1} and \eqref{commut2} have supports disjoint from those of $\rr_7$ and $\rs_7,$ respectively.
Thus, the quadratic relation space $\ker(\C_{i+1}^{i}\otimes\C_{i+2}^{i+1}\to\C_{i+2}^{i})$ has dimension eight for every $i\in\ZZ.$
By Lemma \ref{lem:quantum-veronese-pbw}, the algebra $\cS_{\vartheta,J}$ is quadratic, and its quadratic relation spaces have dimension $18-\dim(S_\vartheta)_3=18-10=8.$
The maps \eqref{eq:plane-map} therefore identify the quadratic relation spaces and extend to an isomorphism $\C\cong\cS_{\vartheta,J}.$
Moreover, the formulas in \eqref{eq:plane-map} agree at indices $j$ and $j+2,$ and $\rho(j+2)=\rho(j)+3.$
Hence, under this isomorphism, $\theta$ corresponds to the shift $(a,b)\mapsto(a+3,b+3)$ of both indices of $\cS_\vartheta.$
Replacing $\zeta$ by $\zeta^{-1}$ proves the assertion for $\underline\C.$
\end{proof}

For the identifications of graded section algebras, we use $S_\vartheta$ with relations \eqref{eq:quantum-plane-rules}.

\begin{corollary}\label{thm:quotient-pbw}\label{cor:quotient-dimensions}
The quadratic $\ZZ$\!--algebras $\C$ and $\underline\C$ are PBW and hence Koszul.
For $m\ge0,$ the components of $\C$ have dimensions
\[
 \dim\C_{j+2m}^{j}=\binom{3m+2}{2},\qquad
 \dim\C_{j+2m+1}^{j}=
 \begin{cases}\binom{3m+3}{2}&j\text{ even},\\
 \binom{3m+4}{2}&j\text{ odd}.
 \end{cases}
\]
The same formulas hold for $\underline\C.$
\end{corollary}
\begin{proof}
The PBW and Koszul properties follow from Theorem \ref{thm:quantum-plane} and Lemma \ref{lem:quantum-veronese-pbw}.
The dimensions follow from Proposition \ref{prop:quantum-polynomial}, since $\rho(j+2m)-\rho(j)=3m$ and $\rho(j+2m+1)-\rho(j)$ equals $3m+1$ or $3m+2$ according to the parity of $j.$
\end{proof}

\subsection{The Hilbert series}
We compute the Hilbert matrix of $\L$ from Koszulness and the dimensions of its quadratic dual.

\begin{corollary}\label{cor:hilbert-A}
Define the Hilbert matrices by
\[
 H_{\A}(z)_{ab}
 =\sum_{\substack{n\ge0\\n\equiv b-a\ (2)}}
       \dim\A_{a+n}^{a}\,z^n,\qquad
 a,b\in\{0,1\},\quad \A=\L,\C.
\]
For the Hilbert matrix $\widehat H(z)$ of $\widehat\L$ computed in \eqref{eq:dual-hilbert}, we have
\begin{equation}\label{eq:hilbert-A}
\begin{gathered}
 H_{\L}(z)=\widehat H(z)/(1-z^2)^4,\qquad
 H_{\C}(z)=\widehat H(z)/(1-z^2)^3,\\
 \widehat H(-z)\widehat H(z)=(1-z^2)^4I_2,
\end{gathered}
\end{equation}
where $I_2$ is the $2\times2$ identity matrix.
Moreover, $\widetilde R/(\theta^{-1}(t_0))\cong S_\vartheta^{(3)},$ where $\theta^{-1}(t_0)\in \widetilde R_1.$
\end{corollary}
\begin{proof}
By Theorem \ref{thm:koszul}, the minimal projective resolutions of the simple right modules $\mS_i$ are the Koszul resolutions.
For $N\ge0,$ the components at index $i+N$ of the terms in the Koszul resolution of $\mS_i$ form an exact sequence of finite-dimensional vector spaces.
Taking the alternating sum of dimensions in this sequence gives
\[
 \sum_{n=0}^{N}(-1)^n
 \dim\widehat\L_i^{i+n}\,
 \dim\L_{i+N}^{i+n}
 =
 \begin{cases}1&N=0,\\0&N>0.\end{cases}
\]
Two-periodicity turns these identities into $\widehat H(-z)H_\L(z)=I_2.$
Using \eqref{eq:dual-hilbert}, we obtain
\[
 \widehat H(-z)\widehat H(z)=(1-z^2)^4I_2.
\]
Since $\widehat H(-z)$ has constant term $I_2,$ it is invertible as a matrix of formal power series.
This proves the formula for $H_\L$ in \eqref{eq:hilbert-A}.
The formula for $H_\C$ in \eqref{eq:hilbert-A} follows from Corollary \ref{cor:quotient-dimensions}.
The last assertion is the restriction to even indices of Theorem \ref{thm:quantum-plane}, since two successive steps in $J$ have total weight three.
\end{proof}

\subsection{Regularity and normal bases}
We prove regularity of the normalizing sequences by comparing the dimensions of $\L$ and its quotients, and then construct a basis of $\L.$

For a path in $\L_{j+n}^{j},$ write $t^h=t_jt_{j+2}\cdots t_{j+2h-2},$ with $t^0=e_j.$
The expanded index sequence of a word in the degree-one generators is obtained by replacing each $x_a$ by $(a)$ and each $u^{ab}$ by $(a,b),$ with $a\le b,$ and concatenating these sequences from left to right.
Bars are suppressed using $\theta.$

\begin{lemma}\label{lem:normal-basis}
Both normalizing sequences $\mathcal T$ and $\underline{\mathcal T}$ are regular on the left and on the right.
For every $j$ and $n\ge0,$ a basis of $\L_{j+n}^{j}$ consists of the composable words $t^h w,$ where $0\le2h\le n$ and $w\in\L_{j+n}^{j+2h}$ is a word in the degree-one generators with nondecreasing expanded indices.
\end{lemma}
\begin{proof}
For $j\in\ZZ$ and $n\ge2,$ normality of $\mathcal T$ gives the right-exact sequences
\[
\begin{aligned}
 \L_{j+n}^{j+2}&\xrightarrow{\,t_j\,\cdot\,}
 \L_{j+n}^{j}\longrightarrow\C_{j+n}^{j}\longrightarrow0,\\
 \L_{j+n-2}^{j}&\xrightarrow{\,\cdot\,t_{j+n-2}\,}
 \L_{j+n}^{j}\longrightarrow\C_{j+n}^{j}\longrightarrow0.
\end{aligned}
\]
By two-periodicity and \eqref{eq:hilbert-A}, the dimension of each source is
$\dim\L_{j+n}^{j}-\dim\C_{j+n}^{j}.$
Both multiplication maps are therefore injective.
The same argument applies to $\underline{\mathcal T},$ since $\underline\C$ has the same component dimensions as $\C.$

Under \eqref{eq:plane-map}, the images of the words $w$ with nondecreasing expanded indices are nonzero scalar multiples of the ordered monomial basis of the corresponding component of $S_\vartheta.$
Their classes therefore form a basis of $\C_{j+n}^{j}.$
The first short exact sequence above now gives the basis $t^h w$ by induction on $n,$ starting with $n=0,1.$
\end{proof}

\subsection{Noetherian property}
We use the Ore subalgebra described in Appendix \ref{app:section-algebra} to prove that $\L,$ $\widetilde R,$ and $R$ are noetherian.

\begin{proposition}\label{prop:noetherian}
The algebras $\L,$ $\widetilde R,$ and $R=\widetilde R^{(2)}$ are noetherian on both sides.
\end{proposition}
\begin{proof}
By Proposition \ref{prop:ore-subalgebra} in Appendix \ref{app:section-algebra}, the algebra $\widetilde R$ is finite free on both sides over the noetherian iterated Ore extension $\Gamma.$
Every ascending chain of right ideals of $\widetilde R$ is a chain of right $\Gamma$-submodules of the noetherian right $\Gamma$-module $\widetilde R,$ so it stabilizes.
The same argument on the left proves that $\widetilde R$ is noetherian on both sides.
Define the graded algebra $\widetilde R_{\mathrm{odd}}=\bigoplus_{m\ge0}\L_1^{1-2m}$ with multiplication $a\star b=\theta^{-n}(a)b$ for homogeneous $a,b$ of degrees $m,n.$
Reversing paths and interchanging $p,q$ gives a graded isomorphism $\widetilde R_{\mathrm{odd}}\cong\widetilde R(q,p)^{\mathrm{op}},$ so $\widetilde R_{\mathrm{odd}}$ is noetherian on both sides.
Put $\B=\L_{2\ZZ}\oplus\L_{2\ZZ+1}.$
The two summands are the $\ZZ$\!--algebras associated with $\widetilde R$ and $\widetilde R_{\mathrm{odd}},$ respectively, so $\B$ is noetherian.
Since $\L$ is generated in degree one, multiplication gives a surjection of right $\B$-modules
\[
 e_i\B\oplus\bigl(\L_{i+1}^{i}\otimes_\kk e_{i+1}\B\bigr)
 \twoheadrightarrow e_i\L,\qquad i\in\ZZ.
\]
The source is a finite direct sum of noetherian projective right $\B$-modules because $\L_{i+1}^{i}$ is finite-dimensional.
Thus, the module $e_i\L$ is noetherian as a right $\B$-module, and hence as a right $\L$-module.
The left-module argument is identical, proving that $\L$ is noetherian on both sides.
Projection to even degrees gives $I\widetilde R\cap R=I$ for every right ideal $I$ of $R,$ and similarly on the left.
For an ascending chain $I_1\subseteq I_2\subseteq\cdots$ of right ideals of $R,$ the chain $I_1\widetilde R\subseteq I_2\widetilde R\subseteq\cdots$ stabilizes because $\widetilde R$ is right noetherian.
Intersecting the extended ideals with $R$ gives $I_n=(I_n\widetilde R)\cap R,$ so the ideals $I_n$ are equal for all sufficiently large $n.$
The same argument applies to left ideals.
Thus, the algebra $R$ is noetherian on both sides.
\end{proof}

\section{Regularity and section algebras}\label{sec:regularity}
In this section, we prove that $\L$ is AS-regular of dimension four.
We also prove that the algebras $\widetilde R$ and $R$ are Koszul and AS-Gorenstein, and compute their Hilbert series.

\subsection{Artin--Schelter regularity}
We deduce Artin--Schelter regularity from Koszulness and the Frobenius pairings of $\widehat\L.$

\begin{theorem}\label{thm:regular}
The $\ZZ$\!--algebra $\L$ is AS-regular of dimension $4$ and Gorenstein parameter $4.$
\end{theorem}
\begin{proof}
By Theorem \ref{thm:koszul} and Corollary \ref{prop:dual}, the Koszul bimodule resolution of $\L$ has length four.
For every right $\L$-module $\mM,$ applying the functor $\mM\otimes_\L(-)$ to this bimodule resolution gives a projective resolution of $\mM$ of length at most four.
Since $\Ext_\L^4(\mS_i,\mS_{i+4})\cong\widehat\L_i^{i+4}\cong\kk,$ the global dimension of $\L$ is four.
Corollary \ref{cor:dual-pairing} and Proposition \ref{prop:bp-frobenius} give the Gorenstein condition with parameter four.
Polynomial growth follows from \eqref{eq:hilbert-A}.
Thus, the algebra $\L$ is AS-regular.
\end{proof}

\subsection{Graded section algebras}
We prove Koszulness and the Gorenstein conditions for $\widetilde R$ and $R,$ and compute their Hilbert series.

\begin{proposition}\label{prop:section-algebras}
The graded algebras $\widetilde R$ and $R=\widetilde R^{(2)}$ are quadratic and Koszul.
They are AS-Gorenstein of injective dimension $4,$ with Gorenstein parameters $2$ and $1,$ respectively.
For every $n\ge0,$ the component dimensions are
\begin{equation}\label{eq:section-dimensions}
 \dim {\widetilde R}_n=3/2(n+1)^3-1/2(n+1),\qquad
 \dim R_n=12n^3+18n^2+8n+1.
\end{equation}
In particular, both algebras have polynomial growth, and their Hilbert series are
\begin{equation}\label{eq:section-hilbert}
 H_{\widetilde R}(z)=(1+7z+z^2)/(1-z)^4,\qquad
 H_R(z)=(1+35z+35z^2+z^3)/(1-z)^4.
\end{equation}
\end{proposition}
\begin{proof}
The proof of \cite[Chapter 3, Proposition 2.2(i)]{PP}, using the Koszul criterion for $\ZZ$\!--algebras in \cite[Chapter 4, Definition--Proposition 9.2]{PP}, shows that the reindexed restrictions $\L_{2\ZZ}$ and $\L_{4\ZZ}$ are Koszul.
The periodicities $\theta$ and $\theta^2$ identify these restrictions with the $\ZZ$\!--algebras associated with $\widetilde R$ and $R,$ respectively.
Hence, both graded algebras are Koszul and quadratic.
For the Hilbert series, the coefficients of $z^{2n}$ in $(H_\L(z))_{00}$ from \eqref{eq:hilbert-A} give
$H_{\widetilde R}(z)=(1+7z+z^2)/(1-z)^4.$
The coefficient of $z^n$ in this series is $3/2(n+1)^3-1/2(n+1).$
Since $R_n=\widetilde R_{2n},$ we obtain \eqref{eq:section-dimensions} and the formula for $H_R(z)$ in \eqref{eq:section-hilbert}.

For the Gorenstein conditions, Proposition \ref{prop:quantum-polynomial} shows that $S_\vartheta$ is noetherian and AS-regular of dimension three and Gorenstein parameter three.
By \cite[Theorem 3.6(1)]{JZ}, its third Veronese subalgebra $D=S_\vartheta^{(3)}$ is AS-Gorenstein of injective dimension three and parameter one.

Put $t=\theta^{-1}(t_0)\in\widetilde R_1,$ so that $D=\widetilde R/(t).$
Let $\sigma$ be the graded automorphism of $\widetilde R$ defined by $ta=\sigma(a)t$ for $a\in\widetilde R.$
The automorphism $\sigma$ preserves the augmentation $\widetilde R\to\widetilde R_0=\kk,$ so twisting by $\sigma$ does not change the action of $\widetilde R$ on the trivial module $\kk.$
The graded Rees lemma \cite[Section 3.4]{Lv}, applied to the regular normal element $t,$ gives isomorphisms of graded vector spaces
\[
 \underline{\Ext}_{\widetilde R}^{q+1}(\kk,\widetilde R)
 \cong\underline{\Ext}_{D}^{q}(\kk,D)(1),\qquad q\ge0.
\]
Also, we have $\Hom_{\widetilde R}(\kk,\widetilde R)=0,$ since $t$ is regular.
Applying the graded Rees lemma to $\widetilde R^{\mathrm{op}}$ gives
\[
 \underline{\Ext}_{\widetilde R^{\mathrm{op}}}^{q}(\kk,\widetilde R)
 \cong
 \begin{cases}
 \kk(2)&q=4,\\
 0&q\ne4.
 \end{cases}
\]
Moreover, \cite[Theorem 3.6(1)]{Lv} gives injective dimension four on both sides.
Thus, the algebra $\widetilde R$ is AS-Gorenstein with parameter two.
Applying \cite[Theorem 3.6(1)]{JZ} once more shows that $R=\widetilde R^{(2)}$ is AS-Gorenstein of injective dimension four and parameter one.
In both applications, the Veronese index is invertible in $\kk$ and divides the Gorenstein parameter.
For the gradings $D_n=(S_\vartheta)_{3n}$ and $R_n=\widetilde R_{2n},$ the Gorenstein parameters of $D$ and $R$ are equal to $1$.
Graded and ungraded injective dimensions agree for these positively graded noetherian algebras by \cite[Section 3.3, Lemma]{Lv}.
\end{proof}

\section{The Fano minifold and its anticanonical algebra}\label{sec:minifold}

Let $\D=\mathcal D(\QGr\L)^c=\qprf\!\operatorname{-}\L$ with its natural DG enhancement.
Write $\barP_i=\pi(e_i\L)\in\D,$ where $\pi:\Md\L\to\QGr\L$ is the quotient functor defined in Section \ref{subsec:z-algebras-tails}.

\Needspace{6\baselineskip}
\begin{theorem}\label{thm:minifold}
The category $\D$ is equivalent to the category of perfect complexes on a smooth proper three-dimensional noncommutative Fano minifold $\mathsf{M}_{72}(p,q).$
The objects
\[
 \mathcal P=(\barP_3,\barP_2,\barP_1,\barP_0)
\]
form a full strong exceptional collection and generate a geometric helix.
Its Euler matrix is
\begin{equation}\label{eq:euler-matrix}
 \Apr=\begin{pmatrix}
 1&3&11&18\\0&1&6&11\\0&0&1&3\\0&0&0&1
 \end{pmatrix}.
\end{equation}
There are isomorphisms $S(\barP_i)\simeq \barP_{i+4}[3]$ under which $S(f)$ corresponds to $\theta^2(f)[3]$ for every morphism $f:\barP_i\to \barP_j.$
The anticanonical algebra with respect to $\barP_0$ is $R={\widetilde R}^{(2)}.$
The automorphism $T_*$ induced by $T=S^{-1}[3]$ is maximally unipotent on $K_0(\D)\otimes\mathbb Q,$ and the Euler degree of $\mathsf M_{72}(p,q)$ is $72.$
\end{theorem}
\begin{proof}
By Theorem \ref{thm:regular} and Proposition \ref{prop:noetherian}, $\L$ is geometric of period four and noetherian.
Theorem \ref{thm:geometric-tails} gives the full strong exceptional collection $(\barP_3,\barP_2,\barP_1,\barP_0)$ and the geometric helix $(\barP_{-i})_{i\in\ZZ}$ in $\D,$ whose directed $\ZZ$\!--algebra is $\L.$
It also gives $\D\simeq\operatorname{Perf}(B),$ where $B=\bigoplus_{0\le i,j\le3}\L_i^j.$
The algebra $B$ is finite-dimensional and directed, with semisimple diagonal subalgebra $B_0=\bigoplus_{i=0}^3\kk e_i\cong\kk^4.$
The normalized bar resolution of $B$ relative to $B_0$ has length at most three.
Thus, the algebra $B,$ viewed as a DG algebra concentrated in degree zero, is a smooth proper noncommutative scheme $\mathsf M_{72}(p,q)$ with $\operatorname{Perf}(\mathsf M_{72}(p,q))\simeq\D.$
By Proposition \ref{prop:helix-fano}, this scheme is a three-dimensional noncommutative Fano variety.

Put $\varepsilon'_i=(-1)^i\varepsilon_i.$
Equation \eqref{eq:graded-trace} becomes
\(
 \varepsilon'_i(ba)=\varepsilon'_j\bigl(\widehat\theta^{\,2}(a)b\bigr).
\)
Thus, the canonical Frobenius periodicity of \cite[Section 5, formula (5.1)]{BP} is $\widehat\theta^{\,2}.$
Let $\widehat{\mS}_i$ be the simple right $\widehat\L$-module at index $i.$
Under the Koszul equivalence between $\D$ and the stable category of finite-dimensional right $\widehat\L$-modules described in \cite[Propositions 5.2--5.3]{BP}, the object $\barP_i$ corresponds to $\widehat{\mS}_i[-i].$
The Nakayama functor on finite-dimensional right $\widehat\L$-modules is $\nu(\mM)=\mM\otimes_{\widehat\L}\widehat\L^{\vee},$ where $(\widehat\L^{\vee})_i^j=(\widehat\L_j^i)^*$ defines the dual bimodule.
The functor $\nu$ is induced by the canonical Frobenius periodicity and sends $\widehat{\mS}_i$ to $\widehat{\mS}_{i+4}.$
The Serre functor of the stable category of finite-dimensional right $\widehat\L$-modules is $\nu[-1]$ by \cite[Lemma 4.3]{GKO}.
We therefore have
\[
 \nu(\widehat{\mS}_i[-i])[-1]=\widehat{\mS}_{i+4}[-i-1]
 =\widehat{\mS}_{i+4}[-(i+4)][3].
\]
This gives isomorphisms $S(\barP_i)\simeq \barP_{i+4}[3].$
The Koszul equivalence identifies composition in the helix $\ZZ$\!--algebra $\L$ with Yoneda multiplication on extension groups of right $\widehat\L$-modules.
By quadratic duality, the induced action is $(\widehat\theta^{-2})^*=\theta^2,$ so under these isomorphisms $S(f)$ corresponds to $\theta^2(f)[3]$ for every morphism $f:\barP_i\to \barP_j.$
Consequently, the functor $T=S^{-1}[3]$ acts by $T(\barP_i)=\barP_{i-4}$ and $T(f)=\theta^{-2}(f).$
Restriction of scalars along $\theta$ preserves torsion modules and induces an autoequivalence $\widetilde T$ of $\D$ satisfying $\widetilde T(\barP_i)=\barP_{i-2}$ and $\widetilde T(f)=\theta^{-1}(f).$
Thus, we have $\widetilde T^{\,2}\simeq T$ and $\widetilde R_m=\Hom(\barP_0,\widetilde T^{\,m}\barP_0).$
It follows that
\[
 \Hom(\barP_0,T^m\barP_0)
 =\L_0^{-4m}={\widetilde R}_{2m}=R_m,
 \qquad f\star g=\theta^{-2n}(f)g,\qquad g\in R_n.
\]
The multiplication is that of \eqref{eq:section-algebra}, so this identifies the anticanonical algebra with $R.$
Restriction to $4\ZZ$ also gives $\QGr\L\simeq\QGr R$ by Proposition \ref{prop:veronese-tails}.

The helix identification and Corollary \ref{cor:hilbert-A} give the Euler matrix \eqref{eq:euler-matrix}.
Put $N=-\Apr^{-\mathsf t}\Apr-I.$
Using the matrix $\Apr$ in \eqref{eq:euler-matrix}, we obtain $N^4=0.$
By geometricity and \eqref{eq:section-dimensions}, the function $\chi(\barP_0,T^m\barP_0)=\dim R_m$ is a polynomial of degree three in $m,$ so $N^3\ne0.$
Since $\operatorname{rank}K_0(\D)=4,$ the automorphism $T_*$ is maximally unipotent.
The primitive vector $v=(-1,2,-2,1)^{\mathsf t}$ generates $\ker(\Apr+\Apr^{\mathsf t})$ and satisfies $(\Apr v)_4=1.$
By \cite[Lemma 3.6, formula (3.13)]{LO}, we have
\[
 \dim R_m=D(\Apr)m^3/6+O(m^2).
\]
Comparing with \eqref{eq:section-dimensions}, we obtain $D(\Apr)=6\cdot12=72.$
\end{proof}

\appendix
\section{Calculations}\label{sec:calculations}

We check the generator assignments \eqref{eq:dual-map} used in Proposition \ref{prop:dual-veronese}.
In the following calculations, $\psi_i^j$ denotes the substitution on composable tensor words with reversed multiplication, before passage to the quotient.
For example, the relations
$\upsilon^{13}\xi_3=-p\upsilon^{33}\xi_1$ and
$\br\xi_2\upsilon^{12}=-p\br\xi_1\upsilon^{22}$ become
$-q\nu^3y_1y_3^2=-py_1y_3^2$ and
$-q\nu y_1y_2^2=-p\nu^{-2}y_1y_2^2.$
For the first two tensors in each of the first three rows of \eqref{anticommut1}, substitution of \eqref{eq:dual-map} with reversed multiplication and rewriting the resulting products as scalar multiples of $y_1^a y_2^b y_3^c$ using the first three relations in \eqref{eq:truncated-quantum} give the following images under $\psi_0^2,$ in the same order:
\[
 (p-q\nu^3)y_1y_3^2,\quad
 0,\quad 0,\quad
 (p-q\nu^3)\nu^{-2}y_1^2y_2,\quad
 (p-q\nu^3)\nu^{-2}y_2^2y_3,\quad 0.
\]
For the corresponding tensors in \eqref{anticommut2}, the images under $\psi_1^3$ are
\[
 0,\quad
 (p-q\nu^3)\nu^{-2}y_1y_2^2,\quad
 (p-q\nu^3)\nu^{-2}y_2y_3^2,\quad
 0,\quad 0,\quad (p-q\nu^3)y_1^2y_3.
\]
All twelve images vanish because $p-q\nu^3=0.$
The last two tensors in each of \eqref{anticommut1} and \eqref{anticommut2} map to zero under $\psi_0^2$ and $\psi_1^3,$ respectively, because
\[
\begin{aligned}
 \psi_0^2(\upsilon^{23}\xi_1)=\psi_0^2(\upsilon^{13}\xi_2)=\psi_0^2(\upsilon^{12}\xi_3)&=-q\nu y_1y_2y_3,\\
 \psi_1^3(\br\xi_1\upsilon^{23})=\psi_1^3(\br\xi_2\upsilon^{13})=\psi_1^3(\br\xi_3\upsilon^{12})&=-q\nu y_1y_2y_3.
\end{aligned}
\]
We also have
\[
 \psi_0^2(\upsilon^{ii}\xi_i)=\psi_1^3(\br\xi_i\upsilon^{ii})=y_i^3=0,\qquad 1\le i\le3.
\]
This verifies all relations in \eqref{anticommut1}--\eqref{anticommut2}.

\section{The section algebra \texorpdfstring{$\widetilde R$}{R tilde}}\label{app:section-algebra}

We give a presentation of $\widetilde R$ and describe it as a finite free module over an iterated Ore extension.
Identify $\widetilde R_m$ with $\L_{2m}^{0}$ by $\theta^m,$ where $\theta$ is the degree-two periodicity defined in Section \ref{subsec:defining-algebra}.
The product is
\begin{equation}\label{eq:positive-section-product}
 a\cdot b=a\,\theta^m(b),\qquad a\in\L_{2m}^{0},\ b\in\L_{2n}^{0}.
\end{equation}
We suppress these shifts in path words and write $t=t_0\in\widetilde R_1.$

Using the fixed identifications $\wb V\cong V$ and $\wb U\cong U$ given by $\br x_i\mapsto x_i$ and $\br u^{ij}\mapsto u^{ij},$ put
\[
 E=(V\otimes U)/\langle \rr_1,\ldots,\rr_7\rangle,
 \qquad \operatorname{pr}:V\otimes U\longrightarrow E.
\]
Here $\operatorname{pr}$ is the quotient map, and the tensors $\rr_i$ and $\rs_i$ are given in \eqref{commut1}--\eqref{trigon1} and \eqref{commut2}--\eqref{trigon2}.
Define the subspace of quadratic relations by
\[
 J=(\operatorname{pr}\otimes\operatorname{pr})
 \bigl(V\otimes\langle \rs_1,\ldots,\rs_7\rangle\otimes U\bigr)
 \subseteq E\otimes E.
\]
Thus, the subspace $J$ is spanned by the images of $x_a\otimes \rs_\ell\otimes u^{bc}$ for $1\le a\le3,$ $1\le\ell\le7,$ and $1\le b\le c\le3.$

\begin{proposition}\label{prop:section-presentation}
There is an isomorphism of graded algebras
\[
 \widetilde R\cong T_\kk(E)/\langle J\rangle,
 \qquad \dim E=11,\qquad \dim J=82,
\]
where $T_\kk(E)$ is the tensor algebra with $E$ in degree one.
A basis of $\widetilde R_1$ consists of
\[
 t=x_1u^{23}-\zeta x_2u^{13},\qquad
 y_{abc}=x_a u^{bc},\qquad 1\le a\le b\le c\le3.
\]
These eleven elements generate $\widetilde R.$
\end{proposition}
\begin{proof}
A path in $\L_{2m}^{0}$ consists of $m$ successive $VU$ blocks.
The relations $\rr_i$ inside a block give the quotient $E,$ while the relations $\rs_i$ between adjacent blocks give the ideal generated by $J.$
Every defining relation of $\L$ is of one of these two types, so this gives the asserted presentation, with multiplication \eqref{eq:positive-section-product}.
By Lemma \ref{lem:normal-basis}, the classes of $x_1\otimes u^{23}-\zeta x_2\otimes u^{13}$ and $x_a\otimes u^{bc}$ with $1\le a\le b\le c\le3$ form a basis of $E.$
Finally, Corollary \ref{cor:hilbert-A} gives $\dim\widetilde R_2=39.$ Hence, we have $\dim J=11^2-39=82.$
\end{proof}

We recall the definition of an Ore extension \cite[Section 1.2]{MR}.
Let $B$ be a $\kk$-algebra, $\sigma:B\to B$ a $\kk$-algebra endomorphism, and $\delta:B\to B$ a $\kk$-linear map satisfying
\[
 \delta(bb')=\sigma(b)\delta(b')+\delta(b)b',\quad b,b'\in B.
\]

\begin{definition}\label{def:ore-extension}
The \emph{Ore extension} $B[z;\sigma,\delta]$ is the $\kk$-algebra containing $B$ which is free as a left $B$-module with basis $1,z,z^2,\ldots$ and whose multiplication is determined by
\[
 zb=\sigma(b)z+\delta(b),\quad b\in B.
\]
When $\delta=0,$ we write $B[z;\sigma].$
\end{definition}

Put $a=y_{111},\; b=y_{222},$ and $c=y_{333}.$

\begin{proposition}\label{prop:ore-subalgebra}
 The subalgebra $\Gamma\subseteq\widetilde R$ generated by $a,b,c,t$ has defining relations
\begin{gather}
 ba=(p/q)^3ab,\quad cb=(p/q)^3bc,\quad ac=(p/q)^3ca,\label{eq:ore-commutation}
 \qquad ta=at,\quad tb=bt,\quad tc=ct.
\end{gather}
It is the iterated Ore extension
\(
 \Gamma\cong\kk[a][b;\sigma_b][c;\sigma_c][t],\) where
 \[ \sigma_b(a)=(p/q)^3a,\quad
 \sigma_c(a)=(q/p)^3a,\quad\sigma_c(b)=(p/q)^3b.
\]
The algebra $\Gamma$ is noetherian on both sides.
The algebra $\widetilde R$ is free as a left and as a right $\Gamma$-module, with basis
\(
 1,\; y_{112},y_{113},y_{122},y_{123},y_{133},y_{223},y_{233},
 \; y_{112}y_{233}.
\)
\end{proposition}
\begin{proof}
For the relation $ba=(p/q)^3ab$ in \eqref{eq:ore-commutation}, apply $\rr_2,\rr_4$ from \eqref{commut1} and $\rs_2,\rs_4$ from \eqref{commut2}.
Cyclic permutation of the indices gives $cb=(p/q)^3bc$ and $ac=(p/q)^3ca.$
Equation \eqref{eq:normal-automorphism} gives their commutation with $t,$ since $\gamma$ fixes $a,b,c.$
For $h,i,j,k\ge0,$ the monomials $t^ha^ib^jc^k$ are distinct basis words of Lemma \ref{lem:normal-basis}, so there are no further relations in $\Gamma.$
Since $p,q\ne0,$ the maps $\sigma_b$ and $\sigma_c$ are automorphisms.
Ore extensions by automorphisms preserve the left and right noetherian property \cite[Theorem 1.2.9]{MR}.
Applying this result successively shows that $\Gamma$ is noetherian on both sides.

By Corollary \ref{cor:hilbert-A}, the quotient $D=\widetilde R/(t)$ is $S_\vartheta^{(3)}.$
The images of $a,b,c$ are $z_1^3,z_2^3,z_3^3$ under \eqref{eq:plane-map}.
The ordered monomial basis makes $D$ free on both sides over $\Gamma/(t).$
A basis is given by the monomials $z_1^iz_2^jz_3^k$ with $0\le i,j,k\le2$ and $i+j+k\equiv0\pmod3.$
By \eqref{eq:plane-map}, the quotient map $\widetilde R\twoheadrightarrow D\cong S_\vartheta^{(3)}$ sends $1,y_{112},y_{113},y_{122},y_{123},y_{133},y_{223},y_{233},y_{112}y_{233},$ respectively, to nonzero scalar multiples of
\[
 1,\quad z_1^2z_2,\quad z_1^2z_3,\quad z_1z_2^2,\quad z_1z_2z_3,\quad z_1z_3^2,\quad z_2^2z_3,\quad z_2z_3^2,\quad z_1^2z_2^2z_3^2.
\]
Write $w_1,\ldots,w_9$ for the nine elements of $\widetilde R$ listed in the proposition, in that order.
Since their images form a basis of $D$ as a left and as a right $\Gamma/(t)$-module, for each homogeneous $u\in\widetilde R_n$ with $n\ge1$ there exist coefficients $g_\ell,h_\ell\in\Gamma_{n-\deg w_\ell}$ for $\deg w_\ell\le n$ such that
\[
 u-\sum_{\deg w_\ell\le n}g_\ell w_\ell\in t\widetilde R_{n-1},\qquad u-\sum_{\deg w_\ell\le n}w_\ell h_\ell\in\widetilde R_{n-1}t.
\]
Induction on $n,$ starting with $\widetilde R_0=\kk w_1$ and using $t\in\Gamma_1,$ proves that $w_1,\ldots,w_9$ generate $\widetilde R$ as a left and as a right $\Gamma$-module.
Reducing a $\Gamma$-linear relation among $w_1,\ldots,w_9$ modulo $t$ shows that all its coefficients lie in $t\Gamma=\Gamma t.$
Regularity of $t$ allows us to cancel a factor $t$ and repeat the argument.
Thus, every coefficient belongs to $\bigcap_{m\ge0}t^m\Gamma=0,$ since $\Gamma$ is nonnegatively graded and $\deg t=1.$
\end{proof}

\end{document}